\documentclass[11pt]{article}

\usepackage{geometry}
\usepackage{hyperref} 

\usepackage{amsmath}
\usepackage{amssymb}
\usepackage{amscd}
\usepackage{textcomp}
\usepackage{dsfont}
\usepackage{mathrsfs}

\long\def\forget#1{}

\newcommand{\Verkuerzung}[2]{#1}

\newcommand{\lang}[1]{\mbox{#1}}

\usepackage{color}
\newcounter{commentcounter}
\newcommand{\comment}[1]{\stepcounter{commentcounter}{\color{red}\textbf{Comment \arabic{commentcounter}.} #1}
\immediate\write16{}
\immediate\write16{Warning: There was still a comment . . . }
\immediate\write16{}}

\newcounter{urscommentcounter}
\newcommand{\urscomment}[1]{\stepcounter{urscommentcounter}{\color{magenta}\textbf{Comment \arabic{urscommentcounter}.} #1}
\immediate\write16{}
\immediate\write16{Warning: There was still a comment . . . }
\immediate\write16{}}

\def\?{\ 
{\bf\color{red}???}\ 
\immediate\write16{}
\immediate\write16{Warning: There was still a question mark . . . }
\immediate\write16{}}

\usepackage{amsthm}
\theoremstyle{plain}
\newtheorem{theorem}{Theorem}[subsection]

\newtheorem{lemma}[theorem]{Lemma}

\newtheorem{proposition}[theorem]{Proposition}

\theoremstyle{definition}
\newtheorem{definition}[theorem]{Definition}
\newtheorem{definition-theorem}[theorem]{Definition-Theorem}
\newtheorem{definition-remark}[theorem]{Definition-Remark}

\newtheorem{remark}[theorem]{Remark}

\theoremstyle{remark}

\usepackage{xy}
\xyoption{all}

\newdir^{ (}{{}*!/-3pt/\dir^{(}}    
\newdir^{  }{{}*!/-3pt/\dir^{}}    
\newdir_{ (}{{}*!/-3pt/\dir_{(}}    
\newdir_{  }{{}*!/-3pt/\dir_{}}

\newcounter{zahl}

\def\theenumi{(\alph{enumi})}

\def\p@enumii{\theenumi}

\newcommand{\DS}{\displaystyle}
\newcommand{\TS}{\textstyle}
\newcommand{\SC}{\scriptstyle}
\newcommand{\SSC}{\scriptscriptstyle}

\newcommand{\cA}{\mathcal{A}}

\newcommand{\cF}{\mathcal{F}}
\newcommand{\cG}{\mathcal{G}}

\DeclareMathOperator{\Aut}{Aut}

\DeclareMathOperator{\Frob}{Frob}
\DeclareMathOperator{\Gal}{Gal}
\DeclareMathOperator{\GL}{GL}

\DeclareMathOperator{\Hom}{Hom}

\DeclareMathOperator{\Isom}{Isom}

\DeclareMathOperator{\Quot}{Frac}

\DeclareMathOperator{\Rep}{Rep}
\DeclareMathOperator{\Res}{Res}

\DeclareMathOperator{\SL}{SL}

\DeclareMathOperator{\Spec}{Spec}
\DeclareMathOperator{\Spf}{Spf}

\DeclareMathOperator{\Var}{V}

\newcommand{\alg}{{\rm alg}}

\DeclareMathOperator{\equi}{equi}
\newcommand{\et}{{\acute{e}t\/}}
\newcommand{\fppf}{{\it fppf\/}}
\newcommand{\fpqc}{{\it fpqc\/}}

\DeclareMathOperator{\id}{\,id}

\newcommand{\sep}{{\rm sep}}

\newcommand{\topol}{{\rm top}}

\DeclareMathOperator{\whtimes}{\mathchoice
            {\wh{\raisebox{0ex}[0ex]{$\DS\times$}}}
            {\wh{\raisebox{0ex}[0ex]{$\TS\times$}}}
            {\wh{\raisebox{0ex}[0ex]{$\SC\times$}}}
            {\wh{\raisebox{0ex}[0ex]{$\SSC\times$}}}}

\renewcommand{\phi}{\varphi}
\renewcommand{\epsilon}{\varepsilon}

\usepackage{amsfonts}
\newcommand{\BOne} {{\mathchoice{\hbox{\rm1\kern-2.7pt l\kern.9pt}}
                              {\hbox{\rm1\kern-2.7pt l\kern.9pt}}
                              {\hbox{\scriptsize\rm1\kern-2.3pt l\kern.4pt}}
                              {\hbox{\scriptsize\rm1\kern-2.4pt l\kern.5pt}}}}

\newcommand{\BA}{{\mathbb{A}}}

\newcommand{\BD}{{\mathbb{D}}}

\newcommand{\BF}{{\mathbb{F}}}

\newcommand{\BL}{{\mathbb{L}}}

\newcommand{\BN}{{\mathbb{N}}}

\newcommand{\BP}{{\mathbb{P}}}
\newcommand{\BQ}{{\mathbb{Q}}}

\newcommand{\BZ}{{\mathbb{Z}}}

\newcommand{\CA}{{\cal{A}}}

\newcommand{\CE}{{\cal{E}}}
\newcommand{\CF}{{\cal{F}}}
\newcommand{\CG}{{\cal{G}}}

\newcommand{\CI}{{\cal{I}}}

\newcommand{\CL}{{\cal{L}}}
\newcommand{\CM}{{\cal{M}}}
\newcommand{\CN}{{\cal{N}}}
\newcommand{\CO}{{\cal{O}}}
\newcommand{\CP}{{\cal{P}}}
\newcommand{\CQ}{{\cal{Q}}}

\newcommand{\CS}{{\cal{S}}}
\newcommand{\CT}{{\cal{T}}}
\newcommand{\CU}{{\cal{U}}}
\newcommand{\CV}{{\cal{V}}}

\newcommand{\CX}{{\cal{X}}}

\newcommand{\CZ}{{\cal{Z}}}

\newcommand{\FG}{{\mathfrak{G}}}

\newcommand{\FL}{{\mathfrak{L}}}
\newcommand{\FM}{{\mathfrak{M}}}

\newcommand{\scrH}{{\mathscr{H}}}

\let\setminus\smallsetminus

\newcommand{\es}{\enspace}

\newcommand{\dual}{^{\SSC\lor}}

\newcommand{\ul}[1]{{\underline{#1}}}
\newcommand{\ol}[1]{{\overline{#1}}}
\newcommand{\wh}[1]{{\widehat{#1}}}
\newcommand{\wt}[1]{{\widetilde{#1}}}

\newcommand{\SSS}{S}
\newcommand{\TTT}{T}

\usepackage{ifthen}

\newcommand{\invlim}[1][]{\ifthenelse{\equal{#1}{}}
{\DS \lim_{\longleftarrow}}
{\DS \lim_{\underset{#1}{\longleftarrow}}}
}

\newcommand{\dirlim}[1][]{\ifthenelse{\equal{#1}{}}
{\DS \lim_{\longrightarrow}}
{\DS \lim_{\underset{#1}{\longrightarrow}}}
}

\newcommand{\dbl}{{\mathchoice{\mbox{\rm [\hspace{-0.15em}[}}
                              {\mbox{\rm [\hspace{-0.15em}[}}
                              {\mbox{\scriptsize\rm [\hspace{-0.15em}[}}
                              {\mbox{\tiny\rm [\hspace{-0.15em}[}}}}
\newcommand{\dbr}{{\mathchoice{\mbox{\rm ]\hspace{-0.15em}]}}
                              {\mbox{\rm ]\hspace{-0.15em}]}}
                              {\mbox{\scriptsize\rm ]\hspace{-0.15em}]}}
                              {\mbox{\tiny\rm ]\hspace{-0.15em}]}}}}
\newcommand{\dpl}{{\mathchoice{\mbox{\rm (\hspace{-0.15em}(}}
                              {\mbox{\rm (\hspace{-0.15em}(}}
                              {\mbox{\scriptsize\rm (\hspace{-0.15em}(}}
                              {\mbox{\tiny\rm (\hspace{-0.15em}(}}}}
\newcommand{\dpr}{{\mathchoice{\mbox{\rm )\hspace{-0.15em})}}
                              {\mbox{\rm )\hspace{-0.15em})}}
                              {\mbox{\scriptsize\rm )\hspace{-0.15em})}}
                              {\mbox{\tiny\rm )\hspace{-0.15em})}}}}

\newcommand{\dotBD}{\vbox{\hbox{\kern2pt\bf.}\vskip-4.5pt\hbox{$\BD$}}}

\DeclareMathOperator{\QIsog}{QIsog}
\DeclareMathOperator{\Nilp}{\CN \!{\it ilp}}

\DeclareMathOperator{\Sets}{\CS \!{\it ets}}

\def\s{\sigma^\ast}

\def\longto{\longrightarrow}
\def\into{\hookrightarrow}

\def\isoto{\stackrel{}{\mbox{\hspace{1mm}\raisebox{+1.4mm}{$\SC\sim$}\hspace{-3.5mm}$\longrightarrow$}}}

\newbox\mybox
\def\arrover#1{\mathrel{
       \setbox\mybox=\hbox spread 1.4em{\hfil$\scriptstyle#1$\hfil}
       \vbox{\offinterlineskip\copy\mybox
             \hbox to\wd\mybox{\rightarrowfill}}}}

\newcommand{\BaseOfD}{\BF}

\newcommand{\genericG}{P}

\newcommand{\Sht}{Sht}

\DeclareMathOperator{\SpaceFl}{\CF\ell}
\newcommand{\tauGlob}{\tau}
\newcommand{\tauLoc}{\hat\tau}
\newcommand{\charsect}{s}

\DeclareMathOperator{\AbSh}{\CA {\it b}-\CS {\it h}}

\begin{document}

\author{Esmail Arasteh Rad\forget{\footnote{Part of this research was carried out while I was visiting Institute For Research In Fundamental Sciences (IPM).}}
and Somayeh Habibi\forget{\footnote{This research was in part supported by a grant from IPM(No. 93510038).}}}

\date{\today}

\title{Local Models For The Moduli Stacks of Global $\FG$-Shtukas\\}

\maketitle

\begin{abstract}

In this article we develop the theory of local models for  the moduli stacks of global $\FG$-shtukas, the function field analogs for Shimura varieties. Here $\FG$ is a smooth affine group scheme over a smooth projective curve. \forget{We first generalize the previous constructions of the moduli stack of global $\FG$-shtukas by introducing a global version of the boundedness condition $\CZ$ on $\FG$-shtukas. Then we develop the theory of local models for these algebraic stacks. This is done according to the following two approaches.}As the first approach, we relate the local geometry of these moduli stacks to the geometry of Schubert varieties inside global affine Grassmannian, only by means of global methods. Alternatively, our second approach uses the relation between the deformation theory of global $\FG$-shtukas and associated local $\BP$-shtukas at certain characteristic places. Regarding the analogy between function fields and number fields, the first (resp. second) approach corresponds to the Beilinson-Drinfeld-Gaitsgory (resp. Rapoport-Zink) local model for (PEL-)Shimura varieties. As an application, we prove the flatness of these moduli stacks over their reflex rings, for tamely ramified group $\FG$. \forget{Furthermore this discussion will establish a conceptual relation between the above approaches.}Furthermore, we introduce the Kottwitz-Rapoport stratification on these moduli stacks and discuss the intersection cohomology of the special fiber.\\

\end{abstract}

\tableofcontents

%
%

\section{Introduction}

Let $\BF_q$ be a finite field with $q$ elements, let $C$ be a smooth projective geometrically irreducible curve over $\BF_q$, and let $\FG$ be a flat affine group scheme of finite type over $C$. A \emph{global $\FG$-shtuka} $\ul\CG$ over an $\BF_q$-scheme $S$ is a tuple $(\CG,\ul\charsect,\tauGlob)$ consisting of a $\FG$-torsor $\CG$ over $C_S:=C\times_{\BF_q}S$, an $n$-tuple of (characteristic) sections $\ul s:=(\charsect_i)_i\in C^n(S)$ and a Frobenius connection $\tauGlob$ defined outside the graphs $\Gamma_{\charsect_i}$ of the sections $\charsect_i$, that is, an isomorphism $\tauGlob\colon\s \CG|_{C_S\setminus \cup_i \Gamma_{\charsect_i}}\isoto \CG|_{C_S\setminus \cup_i \Gamma_{\charsect_i}}$ where $\s=(\id_C \times \Frob_{q,S})^\ast$. 

Philosophically, a global $\FG$-shtuka may be considered as a $\FG$-motive, in the following sense. It admits crystalline (resp. \'etale) realizations at characteristic places (resp. away from characteristic places) which are endowed with $\FG$-actions; see \cite[Section~5.2]{AH_Local}, \cite[Chapter~6]{AH_Global} and also \cite[Chapter~2]{AH_LR}. Consequently, according to Deligne's motivic conception of Shimura varieties \cite{Deligne1,Deligne2}, the moduli stacks of global $\FG$-shtukas can be regarded as the function field analogue for Shimura varieties. Based on this philosophy, they play a central role in the Langlands program over function fields. Hereupon several moduli spaces (resp. stacks) parametrizing families of such objects have been constructed and studied by various authors. Among those one could mention the space of $F$-sheaves $FSh_{D,r}$ which was considered by Drinfeld~\cite{Drinfeld1} and Lafforgue~\cite{Laff} in their proof of the Langlands correspondence for $\FG=\GL_2$ and \ $\FG=\GL_r$ respectively, and which in turn was generalized by Varshavsky's~\cite{Var} moduli stacks $FBun$.  Likewise the moduli stacks $Cht_{\ul\lambda}$ of Ng\^o and Ng\^o Dac~\cite{NgoNgo}, $\CE\ell\ell_{C,\mathscr{D},I}$ of Laumon, Rapoport and Stuhler~\cite{LRS}, and $\AbSh^{r,d}_H$ of Hartl \cite{Har1} and also the moduli stacks $\nabla_n^{\ul\omega}\scrH_D^1(C,\mathfrak{G})$, constructed by  Hartl and the first author, that generalizes the previous constructions; see \cite{AH_Local} and \cite{AH_Global}. 

The above moduli stack $\nabla_n^{\ul\omega}\scrH_D^1(C,\mathfrak{G})$ parametrizes $\FG$-shtukas bounded  by $\ul\omega$ which are equipped with $D$-level structure. Here $\FG$ is a flat affine group scheme of finite type over $C$. The superscript $\ul\omega$ denotes an $n$-tuple of coweights of $\SL_r$ and $D$ is a divisor on $C$.  It can be shown that $\nabla_n^{\ul\omega}\scrH_D^1(C,\mathfrak{G})$ is Deligne-Mumford and separated over $C^n$; see \cite[Theorem~3.14]{AH_Global}. This construction depends on a choice of a faithful representation $\rho:\FG\to \SL_r$. Accordingly in \cite{AH_Global} the authors also propose an intrinsic alternative definition for the moduli stack of global $\FG$-shtukas, in which they roughly replace $\ul\omega$ by an $n$-tuple $\wh Z_\ul\nu$ of certain (equivalence class of) closed subschemes of twisted affine flag varieties and they further refine the $D$-level structure to $H$-level structure, for a compact open subgroup $H\subset \FG (\BA_Q^{\ul\nu})$. Here $\BA_Q^{\ul \nu}$ is the ring of adeles of $C$ outside the fixed $n$-tuple $\ul\nu:=(\nu_i)_i$ of places $\nu_i$ on $C$. For detailed account on $H$-level structures on a global $\FG$-shtuka, we refer the reader to \cite[Chapter~6]{AH_Global}. The resulting moduli stack is denoted by $\nabla_n^{H,\wh Z_{\ul\nu}}\scrH^1(C,\mathfrak{G})^{\ul \nu}$. Note however that the boundedness conditions introduced in \cite{AH_Global} has been established by imposing bounds to the associated tuple of local $\BP_{\nu_i}$-shtukas, that are obtained by means of the global-local functor \ref{G-LFunc}. In other words, the moduli stack of ``bounded'' $\FG$-shtukas was defined only after passing to the formal completion at the fixed characteristic places $\ul \nu:=(\nu_i)$; see Definition \ref{DefGlobaltoLocal}.\\
In this article we first introduce a \emph{global} notion of boundedness condition $\CZ$ (which roughly is an equivalence class of closed subschemes of a global affine Grassmannian, satisfying some additional properties. This accredits us to define the moduli stack of bounded global $\FG$-shtukas $\nabla_n^\CZ\scrH_D^1(C,\FG)$ over n-fold fiber product $C_\CZ^n:=C_{Q_\CZ}^n$ of the \emph{reflex curve} associated to the \emph{reflex field} $Q_\CZ$; see Definition \ref{DefAltGlobalBC}. \forget{\comment{may be: over a fiber product $\ul C_{\CZ}$ of the \emph{reflex curves}, which are associated to $n$-tuple of reflex fields $\ul Q_\CZ$; see Definition \ref{DefGlobalBC}} \comment{clarify the notation after fixing the definition of BC..}} \\
Now we come to the main theme of the present article, the theory of local models for the moduli stack $\nabla_n^{H,\CZ}\scrH^1(C,\mathfrak{G})$. Here it is necessary to assume that $\FG$ is a smooth affine group scheme over $C$. We study the local models for the moduli stacks of global $\FG$-shtukas according to the following two approaches.\\
\noindent 
As the first approach (global approach), we prove that ``global Schubert varieties'' inside the global affine Grassmannian $GR_{\FG,n}$, see Definition \ref{Glob_Grass}, may appear as a local model for the moduli stack of $\FG$-shtukas $\nabla_n^\CZ\scrH_D^1(C,\mathfrak{G})$. We prove this in Theorem \ref{ThmLocalModelI}. 
Notice that, regarding the analogy between number fields and function fields, this construction mirrors the Beilinson, Drinfeld and Gaitsgory construction of the local model for Shimura varieties.\\
\noindent
Notice that, when $\FG$ is constant, i.e. $\FG:=G_0\times_{\BF_q} C$, for a split reductive group $G_0$, our moduli stacks of global $\FG$-shtukas coincide the Varshavsky's moduli stacks of $F$-bundles. In this case the above observation was first formulated by Varshavsky \cite[Theorem 2.20]{Var}. His proof relies on the well-known theorem of Drinfeld and Simpson \cite{DS95}, which assures that a $G_0$-bundle over $C_S$ is Zariski-locally trivial after suitable \'etale base change $S'\to S$. The latter statement essentially follows from their argument about existence of $B$-structure on any $G_0$-bundle, where $B$ is a Borel subgroup of $G_0$. Hence, one might not hope to implement this result in order to treat the present general case. Therefore, in chapter \ref{LMII} we modify Varshavsky's argument and produce a proof which is independent of the theorem of Drinfeld and Simpson.\\

\noindent
As the second approach (local approach), we use the relation between the deformation theory of global $\FG$-shtukas and local $\BP$-shtukas, to relate the (local) geometry of $\nabla_n^{H,\wh Z_{\ul\nu}}\scrH^1(C,\mathfrak{G})^{\ul \nu}$ to the (local) geometry of (a product of certain) Schubert varieties inside twisted affine flag varieties; see Proposition \ref{PropLocalModelHecke} and Theorem \ref{ThmRapoportZinkLocalModel}. To this goal, we basically implement the local theory of global $\FG$-shtukas, developed in \cite{AH_Local}. 
\forget
{
Subsequently, we construct the local model diagram and we use it to compare the local geometry of $\nabla_n^{H,\wh Z_{\ul\nu}}\scrH^1(C,\mathfrak{G})^{\ul \nu}$ with certain affine Schubert varieties at the end of section \ref{SubsectionLMTII}; see Proposition \ref{PropLocalModelHecke} and Theorem \ref{ThmRapoportZinkLocalModel}.

Namely, in Proposition \ref{PropLocalModelHecke} and Theorem \ref{ThmRapoportZinkLocalModel} we construct the local model roof and we use it to compare the (local) geometry of $\nabla_n^{H,\wh Z_{\ul\nu}}\scrH^1(C,\mathfrak{G})^{\ul \nu}$ to a product of certain affine Schubert varieties. }This construction is analogous to the Rapoport-Zink \cite{RZ} construction of local models for PEL-type Shimura varieties. As an application, in Theorem \ref{ThmFlatIntModel}, we prove the flatness of moduli stack $\nabla_n^{H,\wh Z_{\ul\nu}}\scrH^1(C,\mathfrak{G})^{\ul \nu}$ over fiber product of reflex rings, for a parahoric Bruhat-Tits group scheme $\FG$ whose generic fiber $G$ splits over a tamely ramified extension and that $p$ does not divide the order of the fundamental group $\pi_1(G_{der})$. This result can be viewed as a function field analog of a result of Pappas and Zhu for Shimura varieties; see \cite{PZ}. Furthermore, using the local model roof, we introduce the analogue of \emph{Kottwitz-Rapoport stratification} on the special fiber of $\nabla_n^{H,Z_{\ul\nu}}\scrH^1(C,\mathfrak{G})^\ul \nu$ and we observe that the corresponding intersection cohomology complex is of pure Tate type nature.

\forget{one may view a twisted affine flag variety as a moduli space  parmetrizing tuples $(\CL_+, \delta)$ consisting of a torsor $\CL_+$ for the group of positive loops $L^+\BP$ over $S$ together with a trivialization of the associated $LP$-torsor $\CL$. Then using local theory of global $\FG$-shtukas, developed in \cite{AH_Local}, we establish  
local model theorm ~\ref{ThmRapoportZinkLocalModel}, and consequently, we see that a product of certain affine Schubert varieties, may appear as the local model for $\nabla_n^{H,\wh Z_{\ul\nu}}\scrH^1(C,\mathfrak{G})^{\ul \nu}$.  This construction is analogous to the Rapoport-Zink \cite{RZ} construction of the local model for PEL-type Shimura varieties. Using these results we produce the analogue of the Kottwitz-Rapoport stratification on the special fiber of $\nabla_n^{H,Z_{\ul\nu}}\scrH^1(C,\mathfrak{G})^\ul \nu$ and we show that the corresponding intersection cohomology complex is of mixed Tate type nature. }

\noindent
Let us finally mention that our results in this article have number theoretic applications to the  generalizations of the results of Varshavsky \cite{Var} about the intersection cohomology of the moduli stacks of global $G$-shtukas, as well as V.~Lafforgue \cite{VLaff} about Langlands parameterization, to the non-constant reductive case. Additionally, according to the results obtained in \cite{AH_LR}, we expect further applications related to the possible description of the cohomology of affine Deligne-Lusztig varieties.


%
%

\subsection{Notation and Conventions}\label{Notation and Conventions}
Throughout this article we denote by
\begin{tabbing}
$\genericG_\nu:=\FG\times_C\Spec Q_\nu,$\; \=\kill
$\BF_q$\> a finite field with $q$ elements of characteristic $p$,\\[1mm]
$C$\> a smooth projective geometrically irreducible curve over $\BF_q$,\\[1mm]
$Q:=\BF_q(C)$\> the function field of $C$,\\[1mm]

$\BaseOfD$\> a finite field containing $\BF_q$,\\[1mm]

$\wh A:=\BF\dbl z\dbr$\>  the ring of formal power series in $z$ with coefficients in $\BF$ ,\\[1mm]
$\wh Q:=\Quot(\wh A)$\> its fraction field,\\[1mm]

$\nu$\> a closed point of $C$, also called a \emph{place} of $C$,\\[1mm]
$\BF_\nu$\> the residue field at the place $\nu$ on $C$,\\[1mm]

$\wh A_\nu$\> the completion of the stalk $\CO_{C,\nu}$ at $\nu$,\\[1mm]
$\wh Q_\nu:=\Quot(\wh A_\nu)$\> its fraction field,\\[1mm]

$\BD_R:=\Spec R\dbl z \dbr$ \> \parbox[t]{0.79\textwidth}{\Verkuerzung
{
the spectrum of the ring of formal power series in $z$ with coefficients in an $\BaseOfD$-algebra $R$,
}

{}
}
\Verkuerzung
{
\\[1mm]
$\hat{\BD}_R:=\Spf R\dbl z \dbr$ \> the formal spectrum of $R\dbl z\dbr$ with respect to the $z$-adic topology.
}
{}
\end{tabbing}
\Verkuerzung
{
\noindent
}
\noindent
When the ring $R$ is obvious from the context we drop the subscript $R$ from our notation.\\

\noindent
For a formal scheme $\wh S$ we denote by $\Nilp_{\wh S}$ the category of schemes over $\wh S$ on which an ideal of definition of $\wh S$ is locally nilpotent. We  equip $\Nilp_{\wh S}$ with the \'etale topology. We also denote by
\begin{tabbing}
$\genericG_\nu:=\FG\times_C\Spec \wh Q_\nu,$\; \=\kill
$n\in\BN_{>0}$\> a positive integer,\\[1mm]
$\ul \nu:=(\nu_i)_{i=1\ldots n}$\> an $n$-tuple of closed points of $C$,\\[1mm]
$\BA^\ul\nu$\> the ring of integral adeles of $C$ outside $\ul\nu$,\\[1mm]
$\wh A_\ul\nu$\> the completion of the local ring $\CO_{C^n,\ul\nu}$ of $C^n$ at the closed point $\ul\nu=(\nu_i)$,\\[1mm]
$\Nilp_{\wh A_\ul\nu}:=\Nilp_{\Spf \wh A_\ul\nu}$\> \parbox[t]{0.79\textwidth}{the category of schemes over $C^n$ on which the ideal defining the closed point $\ul\nu\in C^n$ is locally nilpotent,}\\[2mm]
$\Nilp_{\BaseOfD\dbl\zeta\dbr}$\lang{$:=\Nilp_{\hat\BD}$}\> \parbox[t]{0.79\textwidth}{the category of $\BD$-schemes $S$ for which the image of $z$ in $\CO_S$ is locally nilpotent. We denote the image of $z$ by $\zeta$ since we need to distinguish it from $z\in\CO_\BD$.}\\[2mm]
$\FG$\> a smooth affine group scheme of finite type over $C$,\\[1mm]
$G$\> generic fiber of $\FG$,\\[1mm]
$\BP_\nu:=\FG\times_C\Spec \wh A_\nu,$ \> the base change of $\FG$ to $\Spec \wh A_\nu$,\\[1mm]
$\genericG_\nu:=\FG\times_C\Spec \wh Q_\nu,$ \> the generic fiber of $\BP_\nu$ over $\Spec Q_\nu$,\\[1mm]
$\BP$\> a smooth affine group scheme of finite type over $\BD=\Spec\BaseOfD\dbl z\dbr$,\\[1mm] 
$\genericG$\> the generic fiber of $\BP$ over $\Spec\BaseOfD\dpl z\dpr$.
\end{tabbing}

\noindent
Let $S$ be an $\BF_q$-scheme and consider an $n$-tuple $\ul s:=(s_i)_i\in C^n(S)$. We denote by $\Gamma_\ul s$ the union $\bigcup_i \Gamma_{s_i}$ of the graphs $\Gamma_{s_i}\subseteq C_S$. \\

\noindent
For an affine closed subscheme $Z$ of $C_S$ with sheaf $\CI_Z$ we denote by $\BD_S(Z)$ the scheme obtained by taking completion along $Z$ and by $\BD_{S,n}(Z)$ the closed subscheme of $\BD_S(Z)$ which is defined by $\CI_Z^n$. Moreover we set $\dot{\BD}_S(Z):=\BD_S(Z)\times_{C_S} (C_S\setminus Z)$.\\

\noindent
We denote by $\sigma_S \colon  S \to S$ the $\BF_q$-Frobenius endomorphism which acts as the identity on the points of $S$ and as the $q$-power map on the structure sheaf. Likewise we let $\hat{\sigma}_S\colon S\to S$ be the $\BaseOfD$-Frobenius endomorphism of an $\BaseOfD$-scheme $S$. We set
\begin{tabbing}
$\genericG_\nu:=\FG\times_C\Spec Q_\nu,$\; \=\kill
$C_S := C \times_{\Spec\BF_q} S$ ,\> and \\[1mm]
$\sigma := \id_C \times \sigma_S$.
\end{tabbing}

\noindent
Let $H$ be a sheaf of groups (for the \'etale topology) on a scheme $X$. In this article a (\emph{right}) \emph{$H$-torsor} (also called an \emph{$H$-bundle}) on $X$ is a sheaf $\CG$ for the \'etale topology on $X$ together with a (right) action of the sheaf $H$ such that $\CG$ is isomorphic to $H$ on a \'etale covering of $X$. Here $H$ is viewed as an $H$-torsor by right multiplication. \\

\begin{definition}\label{ker}
Assume that we have two morphisms $f,g\colon X\to Y$ of schemes or stacks. We denote by $\equi(f,g\colon X\rightrightarrows Y)$ the pull back of the diagonal under the morphism $(f,g)\colon X\to Y\times_\BZ Y$, that is $\equi(f,g\colon X\rightrightarrows Y)\,:=\,X\times_{(f,g),Y\times Y,\Delta}Y$ where $\Delta=\Delta_{Y/\BZ}\colon Y\to Y\times_\BZ Y$ is the diagonal morphism.
\end{definition}

\begin{definition}\label{DefIwahori-Weyl}
 Assume that the generic fiber $\genericG$ of $\BP$ over $\Spec\BaseOfD\dpl z\dpr$ is connected reductive. Consider the base change $\genericG_L$ of $\genericG$ to $L=\BaseOfD^\alg\dpl z\dpr$. Let $S$ be a maximal split torus in $\genericG_L$ and let $T$ be its centralizer. Since $\BaseOfD^\alg$ is algebraically closed, $\genericG_L$ is quasi-split and so $T$ is a maximal torus in $\genericG_L$. Let $N = N(T)$ be the normalizer of $T$ and let $\CT^0$ be the identity component of the N\'eron model of $T$ over $\CO_L=\BaseOfD^\alg\dbl z\dbr$.

The \emph{Iwahori-Weyl group} associated with $S$ is the quotient group $\wt{W}= N(L)\slash\CT^0(\CO_L)$. It is an extension of the finite Weyl group $W_0 = N(L)/T(L)$ by the coinvariants $X_\ast(T)_I$ under $I=\Gal(L^\sep/L)$:
$$
0 \to X_\ast(T)_I \to \wt W \to W_0 \to 1.
$$
By \cite[Proposition~8]{H-R} there is a bijection
\begin{equation}\label{EqSchubertCell}
L^+\BP(\BaseOfD^\alg)\backslash L\genericG(\BaseOfD^\alg)/L^+\BP(\BaseOfD^\alg) \isoto \wt{W}^\BP  \backslash \wt{W}\slash \wt{W}^\BP
\end{equation}
where $\wt{W}^\BP := (N(L)\cap \BP(\CO_L))\slash \CT^0(\CO_L)$, and where $LP(R)=P(R\dbl z\dbr)$ and $L^+\BP(R)=\BP(R\dbl z\dbr)$ are the loop group, resp.\ the group of positive loops of $\BP$; see \cite[\S\,1.a]{PR2}, or \cite[\S4.5]{B-D}, \cite{Ngo-Polo} and \cite{Faltings03} when $\BP$ is constant. Let $\omega\in \wt{W}^\BP\backslash \wt{W}/\wt{W}^\BP$ and let $\BaseOfD_\omega$ be the fixed field in $\BaseOfD^\alg$ of $\{\gamma\in\Gal(\BaseOfD^\alg/\BaseOfD)\colon \gamma(\omega)=\omega\}$. There is a representative $g_\omega\in L\genericG(\BaseOfD_\omega)$ of $\omega$; see \cite[Example~4.12]{AH_Local}. The \emph{Schubert variety} $\CS(\omega)$ associated with $\omega$ is the ind-scheme theoretic closure of the $L^+\BP$-orbit of $g_\omega$ in $\SpaceFl_\BP\whtimes_{\BaseOfD}\BaseOfD_\omega$. It is a reduced projective variety over $\BaseOfD_\omega$. For further details see \cite{PR2} and \cite{Richarz}.  

\end{definition}

\noindent
Finally by an IC-sheaf $IC(\CX)$ on a stack $\CX$ , we will mean the intermediate extension of the constant perverse sheaf $\ol \BQ_\ell$ on an open dense substack $\CX^\circ$ of $\CX$ such that the corresponding reduced stack $\CX_{red}^\circ$ is smooth. The IC-sheaf is normalized so that it is pure of weight zero.

%
%

\section{Preliminaries}

Let $\BF_q$ be a finite field with $q$ elements, let $C$ be a smooth projective geometrically irreducible curve over $\BF_q$, and let $\FG$ be a smooth affine group scheme of finite type over C.

\begin{definition}
We let $\scrH^1(C,\FG)$ denote the category fibered in groupoids over the category of $\BF_q$-schemes, such that the objects over $S$, $\scrH^1(C,\FG)(S)$, are $\FG$-torsors over $C_S$ (also called $\FG$-bundles) and morphisms are isomorphisms of $\FG$-torsors. 
\end{definition}

\begin{remark}\label{RemBun_G}
One can prove that the stack $\scrH^1(C,\FG)$ is a smooth Artin-stack locally of finite type over $\BF_q$. Furthermore, it admits a covering $\{\scrH_\alpha^1\}_{\alpha}$ by connected open substacks of finite type over $\BF_q$.  \forget{For this one only needs to assume that $\FG$ is a flat affine group scheme of finite type over $C$.} The proof for parahoric $\FG$ (with semisimple generic fiber) can be found in \cite[Proposition~1]{Heinloth} and for general case we refer to \cite[Theorem~2.5]{AH_Global}.
\end{remark}

\begin{remark}\label{RemGlobalRep}
There is a faithful representation $\rho\colon\FG\hookrightarrow\GL(\CV)$ for a vector bundle $\CV$ on $C$ together with an isomorphism $\alpha\colon\wedge^\topol\CV\isoto\CO_C$ such that $\rho$ factors through $\SL(\CV):=\ker\bigl(\det\colon\!\!\GL(\CV)\to\GL(\wedge^\topol\CV)\bigr)$ and the quotients $\SL(\CV)/\FG$ and $\GL(\CV)/\FG$ are quasi-affine schemes over $C$. Note that for the existence of such a representation it even suffices to assume that $\FG$ is a flat affine group scheme. For a detailed account, see \cite{AH_Global}[Proposition~2.2].
\end{remark}

\begin{definition}\label{DefD-LevelStr}
Let $D$ be a proper closed subscheme of $C$. A \emph{$D$-level structure} on a $\FG$-bundle $\CG$ on $C_S$ is a trivialization $\psi\colon \CG\times_{C_S}{D_S}\isoto \FG\times_C D_S$ along $D_S:=D\times_{\BF_q}S$. Let $\scrH_D^1(C,\mathfrak{G})$ denote the stack classifying $\FG$-bundles with $D$-level structure, that is, $\scrH_D^1(C,\mathfrak{G})$ is the category fibred in groupoids over the category of $\BF_q$-schemes, which assigns to an $\BF_q$-scheme $S$ the category whose objects are
$$
Ob\bigl(\scrH_D^1(C,\FG)(S)\bigr):=\left\lbrace (\CG,\psi)\colon \CG\in \scrH^1(C,\FG)(S),\, \psi\colon \CG\times_{C_S}{D_S}\isoto \FG\times_C D_S \right\rbrace,
$$ 
and whose morphisms are those isomorphisms of $\FG$-bundles that preserve the $D$-level structure.
\end{definition}

\noindent
Let us recall the definition of the (unbounded ind-algebraic) Hecke stacks.
\begin{definition}\label{Hecke}
For each natural number $n$, let $Hecke_{n,D}(C,\FG)$ be the stack fibered in groupoids over the category of $\BF_q$-schemes, whose $S$ valued points are tuples $\bigl((\CG,\psi),(\CG',\psi'),\ul\charsect,\tauGlob\bigr)$ where
\begin{itemize}
\item[--] $(\CG,\psi)$ and $(\CG',\psi')$ are in $\scrH_D^1(C,\FG)(S)$,
\item[--] $\ul\charsect:=(\charsect_i)_i \in (C\setminus D)^n(S)$ are sections, and
\item[--] $\tau\colon  \CG_{|_{{C_S}\setminus{\Gamma_{\ul\charsect}}}}\isoto \CG_{|_{{C_S}\setminus{\Gamma_{\ul\charsect}}}}'$ is an isomorphism preserving the $D$-level structures, that is, $\psi'\circ\tauGlob=\psi$.
\end{itemize}
If $D=\emptyset$ we will drop it from the notation. Note that forgetting the isomorphism $\tau$ defines a morphism 
\begin{equation}\label{EqFactors}
Hecke_{n,D}(C,\FG)\to \scrH_D^1(C,\FG)\times \scrH_D^1(C,\FG) \times (C\setminus D)^n.
\end{equation}
\end{definition}

\begin{remark}\label{HeckeisQP}

A choice of faithful representation  $\rho\colon  \FG \to \SL(\CV)$ as in Remark~\ref{RemGlobalRep} with quasi-affine (resp.\ affine) quotient $\SL(\CV)\slash \FG$ and coweights $\omega_i=(d,0,\ldots,0,-d)$ of $\SL(\CV)$, induce an ind-algebraic structure on the stack $Hecke_n(C,\FG)$, which is relatively representable over $\scrH^1(C,\FG)\times_{\BF_q}C^n$ by an ind-quasi-projective (resp. ind-projective) morphism; see \cite[Proposition 3.10]{AH_Global}.
\end{remark}

\begin{definition}\label{Glob_Grass}
The \emph{global affine Grassmannian} $GR_n(C,\FG)$ is the stack fibered in groupoids over the category of $\BF_q$-schemes, whose $S$-valued points are tuples $(\CG,\ul\charsect,\epsilon)$, where $\CG$ is a $\FG$-bundle over $C_S$, $\ul\charsect:=(s_i)_i \in C^n(S)$ and $\epsilon\colon \CG|_{C_S\setminus \Gamma_\ul s}\isoto \FG \times_C (C_S\setminus \Gamma_\ul s)$ is a trivialization. Since we fixed the curve $C$ and the group $\FG$, we often drop them from notation and write  $GR_n:=GR_n(C,\FG)$.
\end{definition}

\begin{remark}\label{RemGlobAffGrass}
Notice that the global affine Grassmannian $GR_n$ is isomorphic to the fiber product $Hecke_n(C,\FG)\times_{\CG,\,\scrH^1(C,\FG),\FG}\Spec\BF_q$ under the morphism sending $(\CG,\ul\charsect,\epsilon)$ to $(\FG_S,\CG,\ul\charsect,\epsilon^{-1})$. Hence, after we fix a faithful representation $\rho\colon\FG\into\SL(\CV)$ and coweights $\ul\omega$, as in Remark \ref{RemGlobalRep}, the ind-algebraic structure on $Hecke_n(C,\FG)$, induces an ind-quasi-projective ind-scheme structure on $GR_n$ over $C^n$.
\end{remark}

\noindent
The following proposition explains the geometry of the stack $Hecke_n(C,\FG)$ as a family over $C^n \times \scrH^1(C,\FG)$.

\begin{proposition}\label{GR-Hecke}
Consider the stacks $Hecke_n(C,\FG)$ and $GR_n\times \scrH^1(C,\FG)$ as families over $C^n \times \scrH^1(C,\FG)$, via the projections $(\CG,\CG',\ul s,\tauGlob)\mapsto (\ul s,\CG')$ and $(\wt{\CG},\ul s,\wt\tauGlob)\times \CG'\mapsto (\ul s,\CG')$ respectively. They are locally isomorphic with respect to the \'etale topology on $C^n\times \scrH^1(C,\FG)$.
\end{proposition}

\begin{proof}
The proof proceeds in a similar way as \cite[Lemma~4.1]{Var}, only one has to replace $S$ by $\scrH^1(C,\FG)$ and take an \'etale cover $V\to C\times_{\BF_q} \scrH^1(C,\FG)$ trivializing the universal $\FG$-bundle over $\scrH^1(C,\FG)$ rather than a Zariski trivialization over $S$. 
Also one sets $U=V\times_{\scrH^1(C,\FG)}\ldots\times_{\scrH^1(C,\FG)}V$, $U'=Hecke_n(C,\FG)\times_{C^n\times \scrH^1(C,\FG)}U$, $U''=GR_n\times_{C^n} U$, $V'=V\times_{C\times \scrH^1(C,\FG),\CG'} C\times U'$ and $V''=V\times_{C\times\scrH^1(C,\FG)} C\times U''$.
\end{proof}

\noindent
Now we recall the construction of the (unbounded ind-algebraic) stack of global $\FG$-shtukas.

\begin{definition}\label{Global Sht}
We define the \emph{moduli stack} $\nabla_n\scrH_D^1(C,\mathfrak{G})$ \emph{of global $\FG$-shtukas with $D$-level structure} to be the preimage in $Hecke_{n,D}(C,\FG)$ of the graph of the Frobenius morphism on $\scrH^1(C,\FG)$. In other words
$$
 \nabla_n\scrH_D^1(C,\mathfrak{G}):=\equi(\sigma_{\scrH_D^1(C,\mathfrak{G})} \circ pr_1,pr_2\colon  Hecke_{n,D}(C,\FG)\rightrightarrows \scrH_D^1(C,\FG)),
$$ 
where $pr_i$ are the projections to the first, resp.\ second factor in \eqref{EqFactors}. Each object $\ul{\cG}$ of $\nabla_n\scrH_D^1(C,\FG)(S)$ is called a \emph{global $\FG$-shtuka with $D$-level structure over $S$} and the corresponding sections $\ul\charsect:=(\charsect_i)_i$ are called the \textit{characteristic sections} (or simply \textit{characteristics}) of $\ul \CG$.\\ 
\noindent
More explicitly a global $\FG$-shtuka $\ul\CG$ with $D$-level structure over an $\BF_q$-scheme $S$ is a tuple $(\CG,\psi,\ul\charsect,\tauGlob)$ consisting of a $\FG$-bundle $\CG$ over $C_S$, a trivialization $\psi\colon \CG\times_{C_S}{D_S}\isoto \FG\times_C D_S$, an $n$-tuple of (characteristic) sections $\ul\charsect$, and an isomorphism $\tauGlob\colon  \s \CG|_{C_S\setminus \Gamma_{\ul\charsect}}\isoto \CG|_{C_S\setminus  \Gamma_{\ul\charsect}}$ with $\psi\circ\tauGlob=\s(\psi)$. If $D=\emptyset$ we drop $\psi$ from $\ul\CG$ and write $\nabla_n\scrH^1(C,\mathfrak{G})$ for the \emph{stack of global $\FG$-shtukas}. Sometimes we will fix the sections $\ul s:=(\charsect_i)_i\in C^n(S)$ and simply call $\ul\CG=(\CG,\tauGlob)$ a global $\FG$-shtuka over $S$. The ind-algebraic structure $\dirlim Hecke_n^\ul\omega(C,\FG)$ on the stack $Hecke_n(C,\FG)$ induces an ind-algebraic structure $\dirlim[\ul\omega]\nabla_n^{\ul \omega} \scrH_D^1(C, \FG)$ on $\nabla_n \scrH_D^1(C, \FG)$. 
\end{definition}

\begin{theorem}\label{ThmnHisArtin}
Let $D$ be a proper closed subscheme of $C$. The stack $\nabla_n \scrH_D^1(C, \FG)=\dirlim\nabla_n^{\ul \omega} \scrH_D^1(C, \FG)$ is an ind-algebraic stack (see \cite[Definition 3.13.]{AH_Global}) over $(C\setminus D)^n$ which is ind-separated and locally of ind-finite type. The stacks $\nabla_n^{\ul \omega} \scrH_D^1(C, \FG)$ are Deligne-Mumford. Moreover, the forgetful morphism $\nabla_n \scrH_D^1(C, \FG)\rightarrow \nabla_n \scrH^1(C, \FG)\times_{C^n}(C\setminus D)^n$ is surjective and a torsor under the finite group $\FG(D)$.
\end{theorem}

\begin{proof}
See \cite[Theorem 3.15]{AH_Global}.
\end{proof}

%
%

\section{Analogue Of Beilinson-Drinfeld-Gaitsgory Local Model}\label{LMII}
\setcounter{equation}{0}

Local models for Shimura varieties in the Beilinson-Drinfeld \cite{B-D} and Gaitsgory \cite{Gat} context are in fact constructed as certain scheme-theoretic closures of orbits of positive loop groups inside a fiber product of global affine Grassmannian and a flag variety, which are associated with cocharacters $\lambda\in X_+(T)$, for a maximal split torus $T$ of a constant split reductive group $G$. \\
\noindent
This construction has been generalized to the case of Shimura varieties with parahoric level structure $K\subset G(\BQ_p)$, for which $G$ is non-split (split over tamely ramified extension) by Pappas and Zhu \cite{PZ}.\\ 
\noindent
The function fields analogous construction, for a split reductive group $G$ over $\BF_q$, is well studied by Varshavsky \cite{Var}. In this chapter we generalize his construction of the local model to the case where $\FG$ is a smooth affine group scheme over $C$. Note in particular that this obviously includes the case where $\FG$ is parahoric\forget{(i.e. a smooth group scheme over $C$ with connected fibers and reductive generic fiber)}. To this purpose, we must first introduce the notion of \emph{global boundedness condition}, and correspondingly, the notion of \emph{reflex fields} for the moduli stacks of global $\FG$-shtukas. 


\subsection{Global Boundedness Conditions}\label{SubsecGLobBound}

In this section we establish an \emph{intrinsic global boundedness condition} on the moduli stack $\nabla_n\scrH^1(C,\FG)$ (resp. $Hecke_n(C,\FG)$). The appellation \emph{global boundedness condition} comes from the fact that it bounds 
the relative position of $\s \CG$ and $\CG$ (resp. $\CG'$ and $\CG$) under the isomorphism $\tau$, globally on $C$.  Note further that our proposed definition is intrinsic, in the sense that it is independent of the choice of representation $\rho: \FG \to SL(\CV)$; see \cite[Remark~3.9]{AH_Global}. The corresponding local version of this notion was previously introduced in \cite[Definition~4.8]{AH_Local}.  \\

\noindent
Let us first recall the definition of global loop groups associated with $\FG$.

\begin{definition}\label{DefGlobalLoopGroup}
The group of (positive) loops $\FL_n \FG$ (resp. $\FL_n^+\FG$) of $\FG$ is an ind-scheme (resp. a scheme) representing the functor whose $R$-valued points consist of tuples $(\ul \charsect, \gamma)$ where $\ul\charsect:=(s_i)_i\in C^n(\Spec R)$ and $\gamma\in \FG(\dot{\BD}(\Gamma_\ul s))$(resp. $\gamma\in \FG(\BD(\Gamma_\ul s))$). The projection $(\ul s,\gamma)\mapsto \ul s$ defines morphism $\FL_n\FG\to C^n$, (resp.~ $\FL_n^+\FG\to C^n$).
\end{definition}

\begin{remark}\label{RemarkBeauvilleLaszlo}
Note that by the general form of the descent lemma of Beauville-Laszlo \cite[Theorem 2.12.1]{B-L}, the map which sends $(\CG, \ul s, \epsilon)\in GR_n(S)$ to the triple $(\ul s,\wh\CG:=\CG|_{\BD(\Gamma_\ul s)}, \dot{\epsilon}:=\epsilon|_{\dot{\BD}(\Gamma_\ul s)})$ is bijective. Thus the loop groups $\FL_n\FG$ and $\FL_n^+ \FG$ act on $GR_n$ by changing the trivialization on $\dot{\BD}(\Gamma_\ul s)$.
\end{remark}

\begin{definition}\label{DefAltGlobalBC}
We fix an algebraic closure $Q^\alg$ of the function field $Q:=\BF_q(C)$ of the curve $C$. For a finite field extension $Q\subset K$ with $K\subset Q^\alg$ we consider the normalization $\wt C_K$ of $C$ in $K$. It is a smooth projective curve over $\BF_q$ together with a finite morphism $\wt C_K\to C$.

\begin{enumerate}

\item\label{DefAltGlobalBC_A} 
For a finite extension $K$ as above, we consider closed ind-subschemes $Z$ of $GR_n\times_{C^n} \wt C_K^n$.
We call two closed ind-subschemes $Z_1\subseteq GR_n\times_{C^n} \wt C_{K_1}^n$ and $Z_2\subseteq GR_n\times_{C^n} \wt C_{K_2}^n$ \emph{equivalent} if there is a finite field extension $K_1.K_2\subset K'\subset Q^\alg$ with corresponding curve $\wt C_{K'}$ finite over $\wt C_{K_1}$ and $\wt C_{K_2}$, such that $Z_1 \times_{\wt C_{K_1}^n} \wt C_{K'}^n=Z_2 \times_{\wt C_{K_2}^n} \wt C_{K'}^n$ in $GR_n\times_{C^n} C_{K'}^n$. 

\item\label{DefAltGlobalBC_B}
Let $\CZ=[Z_K]$ be an equivalence class of closed ind-subschemes $Z_K \subseteq GR_n\times_{C^n} \wt C_K^n$ and let $G_\CZ:= \{g \in \Aut(Q^\alg/Q): g^\ast(\CZ)=\CZ\}$. We define the \emph{field of definition} $Q_\CZ$ of $\CZ$ as the intersection of the fixed field of $G_\CZ$ in $Q^\alg$ with all the finite extensions over which a representative of $\CZ$ exists.

\item\label{DefAltGlobalBC_C}
We define a \emph{bound} to be an equivalence class $\CZ:=[Z_K]$ of closed subscheme $Z_K\subset GR_n\times_{C^n} \wt C_K^n$, such that all the ind-subschemes $Z_K$ are stable under the left $\FL_n^+\FG$-action on $GR_n$. The field of definition $Q_\CZ$ (resp. the curve of definition $C_\CZ:=\wt C_{Q_\CZ}$) of $\CZ$ is called the \emph{reflex field} (resp. \emph{reflex curve}) of $\CZ$. 

\item\label{DefAltGlobalBC_D}
Let $\CZ$ be a bound in the above sense. Let $S$ be an $\BF_q$-scheme equipped with an $\BF_q$-morphism $\ul s'=(s_1',\ldots,s_n')\colon S\to C_{\CZ}^n$ and let $s_i\colon S\to C$ be obtained by composing $s_i'\colon S\to C_{\CZ}$ with the morphism $C_{\CZ}\to C$. Consider an isomorphism $\tau:\CG|_{C_S\setminus \Gamma_{\ul s}}\to \CG'|_{C_S\setminus \Gamma_{\ul s}}$ defined outside the graph of the sections $s_i$. Take an \fppf cover $T\to S$ in such a way that the induced \fppf cover $\BD_{T}(\Gamma_{\ul s_T})\to\BD_{S}(\Gamma_\ul s)$ trivializes ${\wh\CG}'$. Here $\ul s_T:=(s_{T,i})$ denotes the $n$-tuple $(s_{T,i})$ of morphisms $s_{T,i}:T \to C$ which is obtained by composing the covering morphism $T\to S$ with the morphism $s_i$. For existence of such trivializations see Lemma \ref{lemfppfTrivialization} below. Fixing a trivialization $\wh\alpha:{\wh\CG}'\times_{\BD(\Gamma_{\ul s})} \BD_T(\Gamma_{\ul s_T})\tilde{\to} \FG \times_C\BD_T(\Gamma_{\ul s_T})$
we obtain a morphism $T\to GR_n$ which is induced by the tuple 
$$
\left(\ul s_T,\wh{\CG}:=\CG|_{\BD_T(\Gamma_{\ul s_T})},\dot{\wh\alpha} \circ \tau|_{\dot{\BD}_T({\Gamma_{\ul s_T}})}: \dot{\wh{\CG}}:=\wh\CG|_{\dot{\BD}_T({\Gamma_{\ul s_T}})}\tilde{\to}\FG \times_C \dot\BD_T({\Gamma_{\ul s_T}})\right);
$$  

see Remark \ref{RemarkBeauvilleLaszlo}. We say that $\tau:\CG|_{C_S\setminus \Gamma_{\ul s}}\to \CG'|_{C_S\setminus \Gamma_{\ul s}}$ is bounded by $\CZ$ if for all representative $Z_K$ of $\CZ$ over $K$  the induced morphism

$$
T\times_{C_Z^n} \wt C_K^n \to GR_n\times_{C^n} \wt C_K^n
$$

factors through $Z_K$. Note that since $Z_K$ is invariant under the left $\FL_n^+\FG$-action, this definition is independent of the choice of the trivialization $\alpha$.

\item\label{DefAltGlobalBC_E}
We say that a tuple $(\CG,\CG',(s_i),\tau)$ in $\left(Hecke_n(C,\FG)\times_{C^n}C_\CZ^n\right)(S)$ is bounded by $\CZ$ if $\tau$ is bounded by $\CZ$ in the above sense. This consequently establishes the boundedness condition on $\FG$-shtukas in $\left(\nabla_n \scrH^1(C, \FG)\times_{C^n}C_\CZ^n\right)(S)$. We denote the corresponding moduli stacks, obtained by imposing the bound $\CZ$, respectively by $Hecke_n^{\CZ}(C,\FG)$ and $\nabla_n^{\CZ} \scrH^1(C, \FG)$. These stacks naturally lie over the n-fold fiber product $C_{\CZ}^n$ of the reflex curve $C_{\CZ}$ over $\BF_q$.

\forget
{
One similarly defines the boundedness condition on a  $\FG$-shtuka $\ul \CG$ in $\left(\nabla_n \scrH^1(C, \FG)\times_{C^n}C_\CZ^n\right)(S)$. We define the stack $Hecke_n^{\CZ}(C,\FG)$ (resp. $\nabla_n^{\CZ} \scrH^1(C, \FG)$) as the stack whose points over a scheme $S$ in $Sch_{C_{\CZ}^n}$ parametrize those objects of $\left(Hecke_n(C,\FG)\times_{C^n}C_\CZ^n\right)(S)$ (resp. $(\nabla_n \scrH^1(C, \FG)\times_{C^n}C_\CZ^n)(S)$) which are bounded by $\CZ$. Notice that these stacks naturally lie over the n-fold fiber product $C_{\CZ}^n$ of the reflex curve $C_{\CZ}$ over $\BF_q$.
}

\end{enumerate}
\end{definition}

\begin{remark}\label{RemAltReflexRing}

In a similar way as explained in \cite[Remark~4.7]{AH_Local} one can compute the reflex field of $\CZ$ in the following concrete sense. We choose a finite extension $K\subset Q^\alg$ of $Q$ over which a representative $Z_K$ of $\CZ$ exists, and for which the inseparability degree $\iota(K)$ of $K$ over $Q$ is minimal. Then the reflex field $Q_\CZ$ equals $K^{G_\CZ}$ and $C_\CZ=C_{K^{G_\CZ}}$. Moreover, let $\wt K$ be the normal closure of $K$. Then $\iota(\wt K) = \iota(K)$ and therefore $\wt K$ is Galois over $Q_\CZ$ with Galois group

$$
\mbox{$\Gal (\wt K/Q_\CZ)\,=\,\{\gamma\in\Aut_{Q}(\wt K)\text{ with }\gamma(Z_{C_\wt K})=Z_{C_\wt Z}\}\,\subset\,\Aut_{Q}(\wt K)$}.
$$
We conclude that $
Q_\CZ\;=\;\bigl\{x\in{\wt K}\colon\gamma(x)=x\text{ for all }\gamma\in\Aut_Q(\wt K)\text{ with }\gamma(Z_{C_\wt K})=Z_{C_\wt K}\bigr\}\,$. As in the local situation we do not know whether in general $\CZ$ has a representative $Z_{Q_\CZ}$ over reflex curve $C_{\CZ}$.

\end{remark}

\begin{lemma}\label{lemfppfTrivialization}
Consider the effective relative Cartier divisor $\Gamma_\ul s$ in $C_S$. Let $\CG$ be a $\FG$-bundle over $C_S$ and set $\wh\CG:=\CG|_{\BD(\Gamma_\ul s)}$. Then there is an \fppf cover $T\to S$ such that the induced morphism $\BD_T(\Gamma_{{\ul s}_T})\to\BD(\Gamma_\ul s)$ is a trivializing cover for $\wh\CG$.
\end{lemma}
\begin{proof}
Let $\CU\to \BD(\Gamma_\ul s)$ be an \'etale covering that trivializes $\wh \CG'$. Consider the closed immersion $\Gamma_\ul s \to \BD_S(\Gamma_\ul s)$ and set $T:=\Gamma_\ul s \times_{\BD(\Gamma_\ul s)} \CU$. 
To see that $T\to S$ is the desired covering notice that the closed immersion $\Gamma_{\ul s_T}\to \BD_{T,n}(\Gamma_{\ul s_T})$ is defined by a nilpotent sheaf of ideal and moreover $\CU \to \BD_S(\Gamma_\ul S)$ is \'etale, hence the natural morphism 
$\Gamma_{\ul s_T} \to \CU$ 
lifts to a morphism $\BD_{T,n}(\Gamma_{\ul s_T})\to \CU$ and consequently to $\BD_{T}(\Gamma_{\ul s_T})\to \CU$. 
\end{proof}
\forget
{
\begin{remark}[about Definition~\ref{DefAltGlobalBC}(d)]\label{RemfppfTrivialization}
Consider the effective relative Cartier divisor $\Gamma_\ul s$ in $C_S$. Let $\CU\to \BD(\Gamma_\ul s)$ be an \'etale covering that trivializes $\wh \CG'$. Consider the closed immersion $\Gamma_\ul s \to \BD_S(\Gamma_\ul s)$ and set $T:=\Gamma_\ul s \times_{\BD(\Gamma_\ul s)} \CU$. 
To see that $T\to S$ is the desired covering notice that the closed immersion $\Gamma_{\ul s_T}\to \BD_{T,n}(\Gamma_{\ul s_T})$ is defined by a nilpotent sheaf of ideal and moreover $\CU \to \BD_S(\Gamma_\ul S)$ is \'etale, hence the natural morphism 
$\Gamma_{\ul s_T} \to \CU$ 
lifts to a morphism $\BD_{T,n}(\Gamma_{\ul s_T})\to \CU$ and consequently to $\BD_{T}(\Gamma_{\ul s_T})\to \CU$.   
\end{remark}

}

\begin{theorem}\label{nHisDM}
Let $D$ be a proper closed subscheme of $C$ and set $D_\CZ:=D\times_C C_\CZ$. The stack $\nabla_n^{\CZ} \scrH_D^1(C, \FG)$ is a Deligne-Mumford stack locally of finite type and separated over $(C_\CZ\setminus D_\CZ)^n$. It is relatively representable over $\scrH^1(C,\FG)\times_{\BF_q} (C_\CZ\setminus D_\CZ)^n$ by a separated morphism of finite type. 
\end{theorem}

\begin{proof}
The forgetful morphism $\nabla_n^{\CZ} \scrH_D^1(C, \FG) \to \nabla_n^{\CZ} \scrH^1(C, \FG) \times_{C_\CZ^n} (C_\CZ\setminus D_\CZ)^n$ is relatively representable by a finite \'etale surjective morphism, see \ref{ThmnHisArtin}. Thus it suffices to prove the statement for $\nabla_n^{\CZ} \scrH^1(C, \FG)$.

Let $S$ be a scheme over $C_\CZ^n$ and let $Z_{C_K}$ be a representative of the bound $\CZ$ for a finite extension $Q\subseteq K\subset Q^\alg$. As in Definition \ref{DefAltGlobalBC}\ref{DefAltGlobalBC_D}, we consider a trivialization of $\wh\CG'$ over an \fppf covering $T\to S$. Recall that this induces a morphism $T\times_{C_Z^n} C_K^n \to GR_n\times_{C^n} C_K^n$. By definition of boundedness condition \ref{DefAltGlobalBC}\ref{DefAltGlobalBC_C}, the
closed subscheme $T\times_{GR_n\times_{C^n} C_K^n} \CZ$ descends to a closed subscheme of $S$.
Consequently $\nabla_n^{\CZ} \scrH^1(C, \FG)$ is a closed substack of $\nabla_n \scrH^1(C, \FG)$ and we may now conclude by Theorem \ref{ThmnHisArtin}.

\end{proof}

\subsection{The Local Model Theorem For $\nabla_n^\CZ\scrH_D^1(C,\FG)$}

The following theorem asserts that global affine Grassmannians (resp. global Schubert varieties inside global affine Grassmannians ) may be regarded as local models for the moduli stacks of global $\FG$-shtukas (resp.  bounded global $\FG$-shtukas). Here $\FG$ is a smooth affine group scheme over $C$. Notice that when $\FG$ is constant, i.e.  $\FG=G_0\times_{\BF_q}C$ for a split reductive group $G_0$ over $\BF_q$, this observation was first recorded by Varshavsky~\cite{Var}. His proof relies on the well-known theorem of Drinfeld and Simpson \cite{DS95} which assures that a $G_0$-bundle over a relative curve over $S$ is Zariski-locally trivial after suitable \'etale base change $S'\to S$.  Note however that this essentially follows from their argument about existence of $B$-structure on $G_0$-bundles, for a Borel subgroup $B$. Hence, one should not hope to implement this result to treat the general case, which we are mainly interested to study in this article. Namely when $\FG$ ramifies at certain places of $C$ (i.e. when $\FG$ is a parahoric group scheme over $C$). Accordingly, in the proof of the following theorem, we modify Varshavsky's method and provide a proof which is independent of the theorem of Drinfeld and Simpson. In the course of the proof we will see that it suffices to assume that $\FG$ is smooth affine group scheme over $C$. Additionally, we will evidently see that it is not possible to weaken this assumption any further.

\begin{theorem}\label{ThmLocalModelI}
Let $\FG$ be a smooth affine group scheme of finite type over $C$ and let $D$ be a proper closed subsecheme of $C$. For any point $y$ in $\nabla_n\scrH_D^1(C,\FG)$ there exist an ind-\'etale neighborhood $U_y$ of $y$ and a roof
\[ 
\xygraph{
!{<0cm,0cm>;<1cm,0cm>:<0cm,1cm>::}
!{(0,0) }*+{U_y}="a"
!{(-1.5,-1.5) }*+{\nabla_n\scrH_D^1(C,\FG)}="b"
!{(1.5,-1.5) }*+{GR_n,}="c"
"a":^{\text{\'et}}"b" "a":_{\text{\'et}}"c"
}  
\]
of ind-\'etale morphisms. In other words the global affine Grassmannian $GR_n$ is a local model for the moduli stack  $\nabla_n\scrH_D^1(C,\FG)$ of global $\FG$-shtukas. Furthermore for a global bound  $\CZ$ as in Definition \ref{DefAltGlobalBC}, with a representative $Z_K\subset GR_n\times_{C^n} C_K^n$ over finite extension $Q\subset K$, \forget{\comment{or may be a global bound as in Definition~\ref{DefGlobalBC} with a representative  $Z_\ul K\subset GR_n\times_{C^n} C_\ul K$}} the pull back of the above diagram induces a roof of \'etale morphisms between $\nabla^{\CZ}_n\scrH_D^1(C,\FG)\times_{C_\CZ^n}C_K^n$ and $Z_K$. \forget{\comment{or may be  $\nabla^{\CZ}_n\scrH_D^1(C,\FG)\times_{C_\CZ^n}C_\ul K^n$ and $Z_\ul K$.} }  
\end{theorem}

\begin{proof}[Proof of Theorem~\ref{ThmLocalModelI}.]
Regarding Theorem \ref{ThmnHisArtin}, we may ignore the $D$-level structure. Since the curve $C$, the parahoric group $\FG$ and the index $n$ (which stands for the number of characteristic sections) are fixed, we drop them from the notation and simply write $\scrH^1=\scrH^1(C,\FG)$, $Hecke=Hecke_n(C,\FG)$ and $\nabla\scrH^1=\nabla_n\scrH^1(C,\FG)$. \\
Let $y'$ be the image of $y$ in $C^n\times\scrH^1$ under the projection sending $(\CG,\CG',\ul s,\tauGlob)$ to $(\ul s,\CG')$.\\ 
According to Proposition~\ref{GR-Hecke}, we may take an \'etale neighborhood $U\to C^n\times \scrH^1$ of $y'$, such that the restriction $U'$ of $Hecke$ to $U$ and the restriction $U''$ of $GR_n \times \scrH^1$ to $U$ become isomorphic. Now, set $\wt U_y:=U' \times_{Hecke} \nabla\scrH^1$ and consider the following commutative diagram

\[ \xygraph{
!{<0cm,0cm>;<1cm,0cm>:<0cm,1cm>::}
!{(.75,0)}*+{Hecke}="a"
!{(-2.25,0)}*+{U'}="b"
!{(.75,-2.25)}*+{C^n\times \scrH^1}="c"
!{(-2.25,-2.25) }*+{U}="d"
!{(-.75,-3)}*+{GR_n\times \scrH^1}="e"
!{(-3.75,-3) }*+{U''}="f"
!{(-.75,-4.5) }*+{GR_n}="h"
!{(-1,-1) }*+{\nabla\scrH^1}="i"
!{(-5.25,-1) }*+{\wt U_y}="j"
"b":^{\et}"a"
"d":^{\et}"c"
"f":^{\et}"e"
"a":"c"
"b":"d"
"e":_{pr_1}"h"
"e":"c"
"f":"d"
"b":^{\cong}"f"
"i":"c"
"j":^{\et\qquad\quad}"i"
"i":"a"
"j":"b"
"j":"d"
"j":@/_4em/_{\phi'}"h"
}  
\]

We pick an open substack of finite type $\scrH_\alpha^1$ that contains the image of $y$ under projection to $\scrH^1$; see Remark~\ref{RemBun_G}. 
After restricting the above diagram to $\scrH_\alpha^1$ and imposing a $D$-level structure, we may obtain the following diagram

\[ 
\xygraph{
!{<0cm,0cm>;<1cm,0cm>:<0cm,1cm>::}
!{(.75,0) }*+{GR_n.}="a"
!{(-2.25,0) }*+{GR_n\times \scrH_{\alpha,D}^1}="b"
!{(-6,0) }*+{U_{\alpha,D}'}="c"
!{(-4.5,0) }*+{U_{\alpha,D}''}="c'"
!{(-6,2.25) }*+{\wt U_{y,_\alpha,D}}="d"
!{(-6,-1.5) }*+{\scrH_{\alpha,D}^1}="e"
!{(-4.5,-1.5) }*+{\scrH_{\alpha,D}^1}="e'"
!{(-7.5,0) }*+{\nabla\scrH_{\alpha,D}^1}="l"
"d":"c":^{\cong}"c'":^{\et\qquad}"b":^{pr_1}"a"
"d":^{\phi'}"a"
"c":_{f}@/_1em/"e"
"c'":^g@/_-1em/"e'"
"e":^{\sigma_{\scrH_\alpha^1}}"e'"
"d":"l"
} 
\]
Here $f$ is the morphism induced by the projection $\pi\colon Hecke\to \scrH^1$ sending $(\CG,\CG',\ul s,\tauGlob)$ to $\CG$ and $g$ is induced by $U''\to GR_n\times \scrH^1$ followed by the projection. 
We may take $D$ enough big such that $\scrH_{\alpha,D}^1$ admits an \'etale presentation $H_{\alpha,D}^1 \to \scrH_{\alpha,D}^1$; see Remark \ref{RemBun_G} and \cite[Lemma~8.3.9, 8.3.10 and 8.3.11]{Beh}. In addition, according to Theorem \ref{ThmnHisArtin} $\nabla\scrH_{\alpha,D}^1 \to \nabla\scrH_\alpha^1$ is \'etale. Consequently, we may replace $\scrH_{\alpha,D}^1$ by the scheme $H_{\alpha,D}^1$. Then define $U_y:=\wt U_{y,\alpha,D}\times_{\scrH_{\alpha,D}^1}H_{\alpha,D}^1$, the theorem now follows from Remarks \ref{RemBun_G} and \ref{RemGlobAffGrass} and the lemma below.

\end{proof}

\begin{lemma}\label{etaleness}
Let $W$, $T$, $Y$ and $Z$ be schemes locally of finite type over $\BF_q$ and let $Z$ be smooth. Assume that we have a morphism $f\colon W \to Z$, and \'etale morphisms $\iota\colon  W\to T\times_{\BF_q}Y$ and $\phi\colon  Y\to Z$. 
Let $g\colon W\to Z$ denote the morphism $\phi \circ pr_2 \circ \iota$, where  $pr_2\colon T\times_{\BF_q}Y\to Y$ is the projection to the second factor. Consider the following diagram
\[
\xymatrix @C+2pc {
V \ar[d] \ar[drr] \\
W \ar[r]_{\iota\quad} \ar@/_/[d]_{\sigma_Z\circ f} \ar@/^/[d]^g & T\times_{\BF_q}Y \ar[r]_{\quad pr_1} & T \\
Z
}
\]
where $V:=\equi( \sigma_Z \circ f, g\colon  W \rightrightarrows Z)$; see Definition~\ref{ker}. Then the induced morphism $V\to T$ is \'etale.
\end{lemma}

\begin{proof}
Let $v\in V$ be a point and let $w\in W$ and $z=\sigma_Z\circ f(w)=g(w)\in Z$, as well as $t\in T$ and $y\in Y$ be its images. Consider an affine open neighborhood $Z'$ of $z$ in $Z$ which admits an \'etale morphism $\pi\colon Z'\to\wt Z$ to some affine space $\wt Z=\BA^m_{\BF_q}=\Spec\BF_q[z_1,\ldots,z_m]$, and consider an affine neighborhood $T'$ of $t$ which we write as a closed subscheme of some $\wt T=\BA^\ell_{\BF_q}$. Replace $Y$ by an affine neighborhood $Y'$ of $y$ contained in $\phi^{-1}(Z')$ and $W$ by an affine neighborhood $W'$ of $w$ contained in $(\sigma_Z\circ f)^{-1}(Z')\cap g^{-1}(Z')\cap \iota^{-1}(T'\times_{\BF_q}Y')$. Then $V':=W'\times_WV=W'\times_{(\sigma_Zf,g),Z'\times Z',\Delta}Z'$ is an open neighborhood of $v$ in $V$. We may extend the \'etale morphism $\iota\colon W'\to T'\times_{\BF_q}Y'$ to an \'etale morphism $\tilde\iota\colon\wt W\to\wt T\times_{\BF_q}Y'$ with $\wt W\times_{\wt T}T'=W'$. We also extend $\pi\circ f\colon W'\to\wt Z$ to a morphism $\tilde f\colon\wt W\to \wt Z$, let $\tilde g:=\pi\,\phi\, pr_2\,\tilde\iota\colon\wt W\to \wt Z$, and set $\wt V:=\equi(\sigma_{\wt Z}\circ\tilde f,\tilde g\colon\wt W\rightrightarrows\wt Z)=\wt W\times_{(\sigma_{\wt Z}\tilde f,\tilde g),\wt Z\times\wt Z,\Delta}\wt Z$. Since $\Delta\colon Z'\to Z'\times_{\wt Z}Z'$ is an open immersion, also the natural morphism 
\[
V'\;\longto\; \wt V\times_{\wt T}T'\;=\;W'\times_{(\sigma_{\wt Z}\tilde f,\tilde g),\wt Z\times\wt Z,\Delta}\wt Z\;=\;W'\times_{(\sigma_Z f,g),Z'\times Z',\Delta}(Z'\times_{\wt Z}Z')
\]
is an open immersion. Since $\wt W$ is smooth over $\wt T$ of relative dimension $m$ and $\wt V$ is given by $m$ equations $\tilde g^*(z_j)-\tilde f^*(z_j)^q$ with linearly independent differentials $d\tilde g^*(z_j)$ the Jacobi-criterion \cite[\S 2.2, Proposition~7]{BLR} implies that $\wt V\to\wt T$ is \'etale. The lemma follows from this.
\end{proof}

\begin{remark}
Notice that the Varshavsky's arguments \cite{Var} about the intersection cohomology of the moduli stacks of (iterated) $G$-shtukas often rely on the following assumption that the Bott-Samelson-Demazure resolution for an affine Schubert variety is (semi-)small. This is the case for split constant reductive groups, but unfortunately for a parahoric group scheme $\FG$, the semi-smallness of the resolution may fail beyond the (co-)minuscule case. \forget{\comment{should write Richarz and ask about counterexample...} }
\end{remark}

%
%

\section{Analogue Of Rapoport-Zink Local Model}\label{LMI}

The PEL-Shimura varieties appear as moduli spaces for abelian varieties (together with additional structures, i.e. polarization, endomorphism and level structure). For such Shimura varieties, Rapoport and Zink \cite[Chapter 3]{RZ} establish their theory of local models. They consider the moduli space whose $S$-valued points parametrizes abelian varieties $X$ (with additional structures) together with a trivialization of the associated  de Rham cohomology. Consequently, one may simultaneously view this space as a torsor on the Shimura variety by forgetting the trivialization, and on the other hand, it also maps to a flag variety which parametrizes filtrations on $Lie(X)\dual$. The later map is formally smooth over its image by Grothendieck-Messing theory. This provides the local model roof, which they use to relate the local properties of the PEL-Shimura varieties and certain Schubert varieties inside flag varieties. In this chapter we investigate the analogous theory over function fields.

\subsection{Preliminaries On Local Theory}
In this section we provide some background material concerning the local theory of global $\FG$-shtukas.

\forget
{
In this section we discuss the local theory of global $\FG$-shtukas. In particular, we introduce local $\BP$-shtukas and we further explain their relation to the deformation theory of global $\FG$-shtukas. 
}

\begin{definition}\label{DefGlobaltoLocal}
Fix an $n$-tuple $\ul\nu:=(\nu_i)_i$ of places on $C$ with $\nu_i\ne\nu_j$ for $i\ne j$. Let $\wh A_\ul\nu$ be the completion of the local ring $\CO_{C^n,\ul\nu}$ of $C^n$ at the closed point $\ul\nu$, and let $\BF_{\ul\nu}$ be the residue field of the point $\ul\nu$. Then $\BF_{\ul\nu}$ is the compositum of the fields $\BF_{\nu_i}$ inside $\BF_q^\alg$, and $\wh A_\ul\nu\cong\BF_{\ul\nu}\dbl\zeta_1,\ldots,\zeta_n\dbr$ where $\zeta_i$ is a uniformizing parameter of $C$ at $\nu_i$. Let the stack $Hecke_{n,D}(C,\FG)^{\ul\nu}\;:=Hecke_{n,D}(C,\FG)\whtimes_{C^n}\Spf \wh A_\ul\nu$ (resp. $\nabla_n\scrH_D^1(C,\FG)^{\ul\nu}\;:=\;\nabla_n\scrH_D^1(C,\FG)\whtimes_{C^n}\Spf \wh A_\ul\nu$ )
be the formal completion of the ind-algebraic stack $Hecke_{n,D}(C,\FG)$ (resp. $\nabla_n\scrH_D^1(C,\FG)$) along $\ul\nu\in C^n$. It is an ind-algebraic stack over $\Spf \wh A_\ul\nu$ which is ind-separated and locally of ind-finite type; see Remark \ref{HeckeisQP} and Theorem \ref{ThmnHisArtin}. As before if $D= \emptyset$ we will drop it from our notation.
\end{definition}

\noindent
Here we recall the following notion of morphism between global $\FG$-shtukas.

\begin{definition}\label{quasi-isogenyGlob}
Consider a scheme $S$ together with characteristic sections $\ul\charsect=(s_i)_i\in C^n(S)$  and let $\ul{\CG}=(\CG,\tau)$ and $\ul{\CG}'=(\CG',\tau')$ be two global $\FG$-shtukas over $S$ with the same characteristics $\charsect_i$. A \emph{quasi-isogeny} from $\ul\CG$ to $\ul\CG'$ is an isomorphism $f\colon\CG|_{C_S \setminus D_S}\isoto \CG'|_{C_S \setminus D_S}$ satisfying $\tau'\sigma^{\ast}(f)=f\tau$, where $D$ is some effective divisor on $C$.
\end{definition}

\noindent
Before introducing the category of local $\BP$-shtukas, the global-local functor and the local boundedness condition, let us recall the following preparatory material.

\bigskip
Let $\BaseOfD$ be a finite field and $\BaseOfD\dbl z\dbr$ be the power series ring over $\BaseOfD$ in the variable $z$. We let $\BP$ be a smooth affine group scheme over $\BD:=\Spec\BaseOfD\dbl z\dbr$ with connected fibers, and we let $\genericG:=\BP\times_\BD\dot{\BD}$ be the generic fiber of $\BP$ over $\dot\BD:=\Spec\BaseOfD\dpl z\dpr$. 

\begin{definition}\label{DefLoopGps}
The \emph{group of positive loops associated with $\BP$} is the infinite dimensional affine group scheme $L^+\BP$ over $\BaseOfD$ whose $R$-valued points for an $\BaseOfD$-algebra $R$ are 
\[
L^+\BP(R):=\BP(R\dbl z\dbr):=\BP(\BD_R):=\Hom_\BD(\BD_R,\BP)\,.
\]
The \emph{group of loops associated with $\genericG$} is the $\fpqc$-sheaf of groups $L\genericG$ over $\BaseOfD$ whose $R$-valued points for an $\BaseOfD$-algebra $R$ are 
\[
L\genericG(R):=\genericG(R\dpl z\dpr):=\genericG(\dot{\BD}_R):=\Hom_{\dot\BD}(\dot\BD_R,\genericG)\,,
\]
where we write $R\dpl z\dpr:=R\dbl z \dbr[\frac{1}{z}]$ and $\dot{\BD}_R:=\Spec R\dpl z\dpr$. It is representable by an ind-scheme of ind-finite type over $\BaseOfD$; see \cite[\S\,1.a]{PR2}, or \cite[\S4.5]{B-D}, \cite{Ngo-Polo}, \cite{Faltings03} when $\BP$ is constant.
Let $\scrH^1(\Spec \BaseOfD,L^+\BP)\,:=\,[\Spec \BaseOfD/L^+\BP]$ (respectively $\scrH^1(\Spec \BaseOfD,L\genericG)\,:=\,[\Spec \BaseOfD/L\genericG]$) denote the classifying space of $L^+\BP$-torsors (respectively $L\genericG$-torsors). It is a stack fibered in groupoids over the category of $\BaseOfD$-schemes $S$ whose category $\scrH^1(\Spec \BaseOfD,L^+\BP)(S)$ consists of all $L^+\BP$-torsors (resp.\ $L\genericG$-torsors) on $S$. The inclusion of sheaves $L^+\BP\subset L\genericG$ gives rise to the natural 1-morphism 
\begin{equation}\label{EqLoopTorsor}
\L\colon\scrH^1(\Spec \BaseOfD,L^+\BP)\longto \scrH^1(\Spec \BaseOfD,L\genericG),~\CL_+\mapsto \CL\,.
\end{equation}
\end{definition}

\begin{definition}
The affine flag variety $\SpaceFl_\BP$ is defined to be the ind-scheme representing the $fpqc$-sheaf associated with the presheaf
$$
R\;\longmapsto\; L\genericG(R)/L^+\BP(R)\;=\;\BP\left(R\dpl z \dpr \right)/\BP\left(R\dbl z\dbr\right).
$$ 
on the category of $\BaseOfD$-algebras; compare Definition~\ref{DefLoopGps}.
\end{definition}

\begin{remark}\label{RemFlagisquasiproj}
 Recall that Pappas and Rapoport \cite[Theorem~1.4]{PR2} show that $\SpaceFl_\BP$ is ind-quasi-projective over $\BaseOfD$, and hence ind-separated and of ind-finite type over $\BaseOfD$. Additionally, they show that the quotient morphism $L\genericG \to \SpaceFl_\BP$ admits sections locally for the \'etale topology. They proceed as follows. When $\BP = \SL_{r,\BD}$, the \fpqc-sheaf $\CF\ell_\BP$ is called the \emph{affine Grassmanian}. It is an inductive limit of projective schemes over $\BF$, that is, ind-projective over $\BF$; see \cite[Theorem~4.5.1]{B-D} or \cite{Faltings03, Ngo-Polo}.  By \cite[Proposition~1.3]{PR2} and \cite[Proposition~2.1]{AH_Global} there is a faithful representation $\BP ֒\to \SL_r$ with quasi-affine quotient. Pappas and Rapoport show in the proof of \cite[Theorem~1.4]{PR2} that $\CF\ell_\BP \to \CF\ell_{\SL_r}$ is a locally closed embedding, and moreover, if $\SL_r /\BP$ is affine, then $\CF\ell_\BP \to \CF\ell_{\SL_r}$ is even a closed embedding and $\CF\ell_\BP$ is ind-projective. Moreover, if the fibers of $\BP$ over $\BD$ are geometrically connected, it was proved by Richarz \cite[Theorem A]{Richarz13} that $\CF\ell_\BP$ is ind-projective if and only if $\BP$ is a parahoric group scheme in the sense of Bruhat and Tits \cite[D\'efinition 5.2.6]{B-T}; see also \cite{H-R}. Note that, in particular, a parahoric group scheme is smooth with connected fibers and reductive generic fiber. 

\end{remark}

\forget{
\begin{remark}
For a closed point $\ul\nu \in C^n$ with $\nu_i\neq\nu_j$ for $i\neq j$, set $A_\ul \nu :=A_{\nu_1}\otimes_{\BF_q} \dots \otimes_{\BF_q} A_{\nu_n}$. Then we have
$$
(\FL_n\FG)_\ul \nu:=\FL_n\FG\times_{C^n}\Spec A_\ul{\nu}= \prod_i LP_\nu,~~~~~~ (\FL_n^+\FG)_{\ul\nu}:=\FL^+\FG\times_{C^n}\Spec A_{\ul\nu}= \prod_i L^+\BP_{\nu_i},
$$
where $LP_\nu(R)=P_\nu(R\dpl z_\nu\dpr)$ and $L^+\BP_\nu(R)=\BP_\nu(R\dbl z_\nu\dbr)$ are respectively the loop group and the group of positive loops of $\BP_\nu$. Moreover, that for any closed point $\ul\nu\in C^n$ with $\nu_i\neq\nu_j$ for $i\neq j$, the bijection induces an isomorphism $(GR_n)_{\ul\nu}:=GR_n\times_{C^n} \Spf \wh A_\ul\nu\cong\prod_i \CF l_{\BP_{\nu_i}}:= \prod_i LP_{\nu_i}/L^+\BP_{\nu_i}$ by \cite[Proposition~3.11]{LaszloSorger}.

\end{remark}

}

Here we recall the definition of the category of local $\BP$-shtukas.

\begin{definition}\label{localSht}

Let $\CX$ be the fiber product 
$$
\scrH^1(\Spec \BaseOfD,L^+\BP)\times_{\scrH^1(\Spec \BaseOfD,L\genericG)} \scrH^1(\Spec \BaseOfD,L^+\BP)
$$ 
of groupoids. Let $pr_i$ denote the projection onto the $i$-th factor. We define the groupoid of \emph{local $\BP$-shtukas} $\Sht_{\BP}^{\BD}$ to be 
$$
\Sht_{\BP}^{\BD}\;:=\;\equi\left(\hat{\sigma}\circ pr_1,pr_2\colon \CX \rightrightarrows \scrH^1(\Spec \BaseOfD,L^+\BP)\right)\whtimes_{\Spec\BaseOfD}\Spf\BaseOfD\dbl\zeta\dbr.
$$
(see Definition~\ref{ker}) where $\hat{\sigma}:=\hat{\sigma}_{\scrH^1(\Spec \BaseOfD,L^+\BP)}$ is the absolute $\BaseOfD$-Frobenius of $\scrH^1(\Spec \BaseOfD,L^+\BP)$. The category $\Sht_{\BP}^{\BD}$ is fibered in groupoids over the category $\Nilp_{\BaseOfD\dbl\zeta\dbr}$ of $\BaseOfD\dbl\zeta\dbr$-schemes on which $\zeta$ is locally nilpotent. We call an object of the category $\Sht_{\BP}^{\BD}(S)$ a \emph{local $\BP$-shtuka over $S$}. 

More explicitly a local $\BP$-shtuka over $S\in \Nilp_{\BaseOfD\dbl\zeta\dbr}$ is a pair $\ul \CL = (\CL_+,\tauLoc)$ consisting of an $L^+\BP$-torsor $\CL_+$ on $S$ and an isomorphism of the associated loop group torsors $\tauLoc\colon  \hat{\sigma}^\ast \CL \to\CL$. 
\end{definition}

Local $\BP$-shtukas can be viewed as function field analogs of $p$-divisible groups. According to this analogy one may introduce the following notion.

\begin{definition}\label{quasi-isogeny L}
A \emph{quasi-isogeny} $f\colon\ul\CL\to\ul\CL'$ between two local $\BP$-shtukas $\ul{\CL}:=(\CL_+,\tauLoc)$ and $\ul{\CL}':=(\CL_+' ,\tauLoc')$ over $S$ is an isomorphism of the associated $L\genericG$-torsors $f \colon  \CL \to \CL'$ satisfying $f\circ\tauLoc=\tauLoc'\circ\hat{\sigma}^{\ast}f$. We denote by $\QIsog_S(\ul{\CL},\ul{\CL}')$ the set of quasi-isogenies between $\ul{\CL}$ and $\ul{\CL}'$ over $S$. 
\end{definition}

\begin{definition}\label{DefFormalTorsor}
Let $\hat{\BP}$ be the formal group scheme over $\hat{\BD}:=\Spf\BaseOfD\dbl z\dbr$, obtained by the formal completion of $\BP$ along $V(z)$. A \emph{formal $\hat{\BP}$-torsor} over an $\BaseOfD$-scheme $S$ is a $z$-adic formal scheme $\hat{\CP}$ over $\hat{\BD}_{S}:=\hat{\BD}\whtimes_{\BaseOfD} S$ together with an action
$\hat{\BP}\whtimes_{\hat{\BD}}\hat{\CP}\rightarrow\hat{\CP}$ of $\hat{\BP}$ on $\hat{\CP}$ such that there is a covering $\hat{\BD}_{S'} \rightarrow \hat{\BD}_{S}$ where ${S'\rightarrow S}$ is an \fpqc-covering
and a $\hat{\BP}$-equivariant isomorphism $
\hat{\CP} \whtimes_{\hat{\BD}_{S}} \hat{\BD}_{S'}\isoto\hat{\BP} \whtimes_{\hat{\BD}}\hat{\BD}_{S'}$. Here $\hat{\BP}$ acts on itself by right multiplication.
Let $\scrH^1(\hat{\BD},\hat{\BP})$ be the category fibered in groupoids that assigns to each $\BaseOfD$-scheme $S$ the groupoid consisting of all formal $\hat{\BP}$-torsors over $\hat{\BD}_S$.

\end{definition}

\begin{remark}\label{RemFormalTorsor}
\begin{enumerate}

\item\label{formal torsor a}
There is a natural isomorphism $\scrH^1(\hat{\BD},\hat{\BP}) \isoto \scrH^1(\Spec \BaseOfD,L^+\BP)$
of groupoids. In particular all $L^+\BP$-torsors for the $\fpqc$-topology on $S$ are already trivial \'etale locally on $S$; see \cite[Proposition~2.4]{AH_Local}.

\item\label{formal torsor b}

Let $\nu$ be a place on $C$ and let $\BD_\nu:=\Spec \wh A_\nu$ and $\hat\BD_\nu:=\Spf \wh A_\nu$. We set $\BP_{\nu}:=\FG\times_C\Spec \wh A_\nu$ and $\hat\BP_{\nu}:=\FG\times_C\Spf \wh A_\nu$. Let $\deg\nu:=[\BF_\nu:\BF_q]$ and fix an inclusion $\BF_\nu\subset \wh A_\nu$. Assume that we have a section $\charsect\colon S\rightarrow C$ which factors through $\Spf \wh A_\nu$, that is, the image in $\CO_S$ of a uniformizer of $\wh A_\nu$ is locally nilpotent. In this case we have $\hat\BD_\nu\whtimes_{\BF_q}\, S\cong \coprod_{\ell\in\BZ/(\deg\nu)} \Var (\mathfrak{a}_{\nu,\ell})\cong \coprod_{\ell\in \BZ/(\deg\nu)}  \hat{\BD}_{\nu,S}$,
\forget{
\begin{equation}\label{EqDecomp}
\hat\BD_\nu\whtimes_{\BF_q}\, S\cong \coprod_{\ell\in\BZ/(\deg\nu)} \Var (\mathfrak{a}_{\nu,\ell})\cong \coprod_{\ell\in \BZ/(\deg\nu)}  \hat{\BD}_{\nu,S},
\end{equation}
}
where $\hat\BD_{\nu,S}:=\hat\BD_\nu\whtimes_{\BF_\nu}S$ and where $\Var (\mathfrak{a}_{\nu,\ell})$ denotes the component identified by the ideal $\mathfrak{a}_{\nu,\ell}=\langle a \otimes 1 -1\otimes \charsect^\ast (a)^{q^\ell}\colon a \in \BF_\nu \rangle$. Note that $\sigma$ cyclically permutes these components and thus the $\BF_\nu$-Frobenius $\sigma^{\deg{\nu}}=:\hat\sigma$ leaves each of the components $\Var (\mathfrak{a}_{\nu,\ell})$ stable. Also note that there are canonical isomorphisms $\Var (\mathfrak{a}_{\nu,\ell})\cong\hat\BD_{\nu,S}$ for all $\ell$.

\item\label{formal torsor c}

 Assume that we have a section $\charsect\colon S\rightarrow C$ which factors through $\Spf \wh A_\nu$ as above. By part \ref{formal torsor b} we may decompose $\CG\whtimes_{C_S}(\Spf \wh A_{\nu_i}\whtimes_{\BF_q}S)\cong\coprod_{\ell\in \BZ/(\deg\nu_i)}\CG\whtimes_{C_S}\Var (\mathfrak{a}_{\nu_i,\ell})$
into a finite product with components $\CG\whtimes_{C_S}\Var (\mathfrak{a}_{\nu_i,\ell})\in\scrH^1(\hat\BD_{\nu_i},\hat\BP_{\nu_i})(S)$. According to \ref{formal torsor a}, we view $\CG\whtimes_{C_S}\Var (\mathfrak{a}_{\nu_i,0})$ as $L^+\BP_\nu$-torsor over $S$, which we denote by $\L_\nu^+\CG$. Furthermore we denote by $\L_\nu \CG$ the $LP_\nu$-torsor associated with $L^+\BP_\nu$-torsor $\L_\nu^+\CG$ regarding the natural 1-morphism (4.1).  

\item\label{formal torsor d}

Fix an $n$-tuple $\ul\nu:=(\nu_i)$ of closed points of $C$. Let $\ul s:=(s_i)$ be $n$-tuple of sections $s_i:S\to C$, such that $s_i$ factors through $\Spf \wh A_{\nu_i}$. We set $\dot{(C_S)}_{\ul s}:=C_S\setminus\Gamma_\ul s$. Let  $[\BF_q/\FG]_e(\dot{(C_S)}_{\ul s})$ denote the category of $\FG$-bundles over $\dot{(C_S)}_{\ul s}$ that can be extended to a $\FG$-bundle over whole relative curve $C_S$. We denote by $\dot{(~)}_\ul s$ the restriction functor 
$$
\dot{(~)}_\ul s\colon\scrH^1(C,\FG)\longto [\BF_q/\FG]_e(\dot{(C_S)}_{\ul s})
$$ 
which assigns to a $\FG$-torsor $\CG$ over $C_S$, the $\FG$-torsor $\dot{(\CG)}_\ul s:=\CG\times_{C_S}\dot{(C_S)}_\ul s$ over $\dot{(C_S)}_\ul s$. Furthermore, for every $i$ there is a functor
$$
\L_{\nu_i}\colon [\BF_q/\FG]_e(\dot{(C_S)}_{\ul s})\longto\scrH^1(\Spec\BF_\nu,L \genericG_{\nu_i})(S)\,
$$
which sends the $\FG$-torsor $\CG'$ over $\dot{(C_S)}_\ul s$, having some extension $\CG$ over $C_S$, to the $L\genericG_\nu$-torsor $\L(\L^+_{\nu_i}(\CG))$ associated with $\L^+_{\nu_i}\CG$ under \eqref{EqLoopTorsor}. One can verify that this assignment is well defined; see \cite[Section~5.1]{AH_Local}. \\
\item\label{formal torsor e}
(Loop group version of Beauville-Laszlo gluing lemma) We define the groupoid $\mathcal{D}(\ul s)$ whose objects consist of the tuples $
(\CG',\CL_{\nu_i}^+, \alpha_i: \L_{\nu_i} \CL_{\nu_i}^+\to \L_{\nu_i}\CG')$, where $\CG'$ lies in the category $[\BF_q/\FG]_e(\dot{(C_S)}_{\ul s})$, $\CL_{\nu_i}^+$ is an $L^+{\BP}_{\nu_i}$-torsor over $S$ and finally an isomorphism $\alpha_i: \L_{\nu_i} \CL_+\to \L_{\nu_i}\CG'$ between associated loop torsors.  One can prove that the functor $\scrH^1(C,\FG)(S) \to  \mathcal{D}(\ul s)$, which sends $\CG$ to $\left(\dot{(\CG)}_\ul s,\L_{\nu_i}^+\CG, \alpha_i:\L_{\nu_i} \L_{\nu_i}^+\CG\to \L_{\nu_i}\dot{(\CG)}_\ul s\right)$ is an equivalence of categories; see \cite[Lemma~5.1]{AH_Local}.

\end{enumerate}
\end{remark}

\subsection{$\BP$-Shtukas And The Deformation Theory Of Global $\FG$-Shtukas}

First of all, we recall an important feature of the morphisms of the category of local $\BP$-shtukas. Namely, as in the theory of $p$-divisible groups, quasi-isogenies between local $\BP$-shtukas enjoy the rigidity property. More explicitly, a quasi-isogeny between local $\BP$-shtukas lifts over infinitesimal thickenings, thanks to the Frobenius connections.

\begin{proposition}[Rigidity of quasi-isogenies for local $\BP$-shtukas] \label{PropRigidityLocal}
Let $S$ be a scheme in $\Nilp_{\BaseOfD\dbl\zeta\dbr}$ and
let $j \colon  \bar{S}\rightarrow S$ be a closed immersion defined by a sheaf of ideals $\CI$ which is locally nilpotent.
Let $\ul{\CL}$ and $\ul{\CL}'$ be two local $\BP$-shtukas over $S$. Then
$$
\QIsog_S(\ul{\CL}, \ul{\CL}') \longto \QIsog_{\bar{S}}(j^*\ul{\CL}, j^*\ul{\CL}') ,\quad f \mapsto j^*f
$$
is a bijection of sets.
\end{proposition}

\begin{proof} See \cite[Proposition 2.11]{AH_Local}.

\end{proof}
\noindent
Notice that, like for abelian varieties, the corresponding statement for global $\FG$-shtukas only holds in fixed finite characteristics. For a detailed account see \cite[Chapter~2 and Chapter~5]{AH_Global}.\\

Analogously to the functor which assigns to an abelian variety over a $\BZ_p$-scheme its $p$-divisible group, there is a global-local functor from the category of global $\FG$-shtukas to the category of local $\BP_{\nu_i}$-shtukas. This functor has been introduced in \cite[Section~5.2]{AH_Local}. As an additional analogy, Hartl and the first author proved that the infinitesimal deformation of a global $\FG$-shtuka is completely ruled by the infinitesimal deformations of the associated local $\BP$-shtukas at the characteristic places. Bellow we explain this phenomenon.

\noindent
We set $\BP_{\nu_i}:=\FG\times_C\Spec \wh A_{\nu_i}$ and $\hat\BP_{\nu_i}:=\FG\times_C\Spf \wh A_{\nu_i}$. Let $(\CG,\ul\charsect,\tauGlob)\in \nabla_n\scrH^1(C,\FG)^{\ul\nu}(S)$, that is, $s_i\colon S\to C$ factors through $\Spf \wh A_{\nu_i}$. According to \ref{RemFormalTorsor}\ref{formal torsor c} we define the \emph{global-local functor} by

\begin{eqnarray}\label{EqGlobLocFunctor}
\wh\Gamma_{\nu_i}(-)\colon \es \nabla_n\scrH^1(C,\FG)^{\ul\nu}(S) & \longto & \Sht_{\BP_{\nu_i}}^{\Spec \wh A_{\nu_i}}(S)\,,\nonumber\\
(\CG,\tauGlob) & \longmapsto & \bigl(\L_{\nu_i}^+\CG,\L_{\nu_i}\tauGlob^{\deg\nu_i}\bigr)\,,\nonumber\\[2mm]
\qquad\ul{\wh\Gamma}\;:=\; \prod_i\wh\Gamma_{\nu_i}\colon \es \nabla_n\scrH^1(C,\FG)^{\ul\nu}(S) & \longto & \prod_i \Sht_{\BP_{\nu_i}}^{\Spec \wh A_{\nu_i}}(S)\, ,\label{G-LFunc}
\end{eqnarray}
where $\L_{\nu_i}\tauGlob^{\deg\nu_i}\colon(\sigma^{\deg\nu_i})^*\L_{\nu_i}\CG\isoto \L_{\nu_i}\CG$ is the $\BF_{\nu_i}$-Frobenius on the loop group torsor $\L_{\nu_i}\CG$ associated with $\L_{\nu_i}^+\CG$. These functors also transform quasi-isogenies into quasi-isogenies. For further explanation see \cite[Section~5.2]{AH_Local}.


Let $S \in \Nilp_{\wh A_\ul\nu}$ and let $j: \ol S \to S$ be a closed subscheme defined by a locally nilpotent sheaf of ideals $\CI$. Let $\ol{\ul\CG}$ be a global $\FG$-shtuka $\nabla_n\scrH^1(C,\FG)^{\ul \nu}(\bar{S})$. We let $Defo_S(\bar{\ul{\cG}})$ denote the category of infinitesimal deformations of $\ol{\ul\CG}$ over $S$. More explicitly $Defo_S(\bar{\ul{\cG}})$ is th category of lifts of $\ol{\ul\CG}$ to $S$, which consists of all pairs $(\ul\CG,\alpha: j^\ast \ul \CG \to \ol{\ul\CG})$ where $\ul\CG$ belongs to $\nabla_n\scrH^1(C,\FG)^{\ul \nu}(S)$, where $\alpha$ is an isomorphism of global $\FG$-shtukas over $S$.\\
Similarly for a local $\BP$-shtuka $\bar{\CL}$ in $\Sht_\BP^\BD(S)$ we define the category of lifts  $Defo_S(\bar{\ul\CL})$ of $\bar{\ul\CL}$ to $S$.

\begin{theorem}\label{Serre-Tate} 
(The Analog of the Serre-Tate Theorem for $\FG$-shtukas) Keep the notation. Let $\bar{\ul{\cG}}:=(\bar{\CG},\bar{\tau})$ be a global $\FG$-shtuka in $\nabla_n\scrH^1(C,\FG)^{\ul \nu}(\bar{S})$. Let $(\ul{\bar\CL}_i)_i=\wh{\ul\Gamma}(\bar{\ul{\cG}})$. Then the functor 
$$
Defo_S(\bar{\ul{\cG}})\longto \prod_i Defo_S(\ul{\bar\CL}_i)\,,\quad \bigl(\ul\CG,\alpha)\longmapsto(\ul{\wh\Gamma}(\ul\CG),\ul{\wh\Gamma}(\alpha)\bigr)
$$ 
induced by the global-local functor (\ref{G-LFunc}), is an equivalence of categories.
\end{theorem}

\bigskip

Analyzing the proof of the above theorem, one can see that the proof essentially relies on the rigidity of quasdi-isogenies, Proposition \ref{PropRigidityLocal}, and the following significant fact. A global $\FG$-shtuka $\ul\CG$ can be pull-backed along a quasi-isogeny $\ul\CL'\to \ul\CL_{\nu_i}$ from a local $\BP_{\nu_i}$-shtuka $\ul\CL'$ to $\ul\CL_{\nu_i}$. Where $\ul\CL_{\nu_i}$ is the local $\BP_{\nu_i}$-shtuka associated with $\ul\CG$ at the characteristic place $\nu_i$, via the global-local functor \ref{G-LFunc}. Here we are not going to explain this phenomenon and we refer the interested reader to \cite[Proposition~5.7]{AH_Local}. However, regarding our goal in this section, we discuss the following simpler case. Namely, the corresponding fact for the moduli stack $\scrH^1(C,\FG)$ of $\FG$-bundles. This will be needed later in the proof of Proposition \ref{PropLocalModelHecke} and the local model theorem \ref{ThmRapoportZinkLocalModel}. \\

Consider the formal scheme $\Spf \BF\dbl\zeta\dbr$ as the ind-scheme $\dirlim\Spec \BF\dbl\zeta\dbr$ and let $\wh\SpaceFl_\BP$ be the fiber product $\SpaceFl_\BP \times \Spf \BF\dbl\zeta\dbr$ in the category of ind-schemes ; see \cite[7.11.1]{B-D}. Thus $\wh\SpaceFl_\BP$ is the restriction of the sheaf $\SpaceFl_\BP$ to the \fppf -site of schemes in $\Nilp_{\BF\dbl\zeta\dbr}$.
The ind-scheme $\wh\SpaceFl_\BP$ pro-represents the functor
\begin{eqnarray}\label{EqRZFunctor}
&(\Nilp_{\BF\dbl\zeta\dbr})^o &\longto  \Sets\hspace{7cm}\vspace{-2mm}\\
&\SSS &\longmapsto  \big\{\text{Isomorphism classes of }(\CL_+,\delta)\colon\;\text{where }\CL_+~\text{is an }\nonumber\\ 
& & ~~~~~~~~~~~~~ \text{$L^+\BP$-torsor over $\SSS$ and a trivialization}~\delta\colon  \CL \to LP_S\nonumber\\
& & ~~~~~~~~~~~~~~~~~~~~~~~~~~~~~~~~\text{of the associated $LP$-torsor}~\CL~\text{over $S$}\big\}.\nonumber 
\end{eqnarray}
\noindent
For details see \cite[Remark 4.2, Theorem 4.4]{AH_Local} and Proposition \ref{PropRigidityLocal}.\\
Note further that for an $L^+\BP$-torsor $\CL_{0,+}$ over $S$, one may further define a twisted variant $\CM(\CL_{0,+})$ of the above functor, which assigns to a scheme $T$ over $S$, the set of isomorphism classes of tuples $(\CL_+,\delta)$ consisting of an $L^+\BP$-torsor $\CL_+$ over $T$ and an isomorphism $\delta\colon  \CL \to \CL_{0,+,T}$.

\forget
{

\begin{remark}\label{RZ space}

Consider the formal scheme $\Spf \BF\dbl\zeta\dbr$ as the ind-scheme $\dirlim\Spec \BF\dbl\zeta\dbr$ and let $\wh\SpaceFl_\BP$ be the fiber product $\wh\SpaceFl_\BP \times \Spf \BF\dbl\zeta\dbr$ in the category of ind-schemes ; see \cite[7.11.1]{B-D}. Thus $\wh\SpaceFl_\BP$ is the restriction of the sheaf $\wh\SpaceFl_\BP$ to the \fppf -site of schemes in $\Nilp_{\BF\dbl\zeta\dbr}$.

The ind-scheme $\wh\SpaceFl_\BP$ pro-represents the functor
\begin{eqnarray}\label{EqRZFunctor}
&(\Nilp_{\BF\dbl\zeta\dbr})^o &\longto  \Sets\hspace{7cm}\vspace{-2mm}\\
&\SSS &\longmapsto  \big\{\text{Isomorphism classes of }(\CL_+,\delta)\colon\;\text{where }\CL_+~\text{is an }\nonumber\\ 
& & ~~~~~~~~~~~~~ \text{$L^+\BP$-torsor over $\SSS$ and a trivialization}~\delta\colon  \CL \to LP_S\nonumber\\
& & ~~~~~~~~~~~~~~~~~~~~~~~~~~~~~~~~\text{of the associated $LP$-torsor}~\CL~\text{over $S$}\big\}.\nonumber 
\end{eqnarray}
\noindent
For details see \cite[Remark 4.2, Theorem 4.4]{AH_Local} and Proposition \ref{PropRigidityLocal}.\\
Note further that for an $L^+\BP$-torsor $\CL_{0,+}$ over $S$, one may further define a twisted variant $\CM(\CL_{0,+})$ of the above functor, which assigns to a scheme $T$ over $S$, the set of isomorphism classes of tuples $(\CL_+,\delta)$ consisting of an $L^+\BP$-torsor $\CL_+$ over $T$ and an isomorphism $\delta\colon  \CL \to \CL_{0,+,T}$.

\end{remark}
}
\noindent
Now we construct the following uniformization map for $\scrH^1(C,\FG)$.

\begin{lemma}\label{LemUnifMap}

Let $S$ be a scheme in $\Nilp_{\wh A_\ul \nu}$. Fix a $\FG$-bundle $\CG$ in $\scrH^1(C,G)(S)$. There is a uniformization map

$$
\Psi_\CG:\prod_{\nu_i} \CM(\L_{\nu_i}^+\CG) \to \scrH^1(C,\FG)|_S.
$$
\noindent
Furthermore this map is formally smooth.

\end{lemma}

\begin{proof}

Suppose that the $\FG$-bundle $\CG$ corresponds to the tuple $\left(\dot{(\CG)}_\ul s, (\L_{\nu_i}^+\CG, \alpha_i:\L_{\nu_i}\L_{\nu_i}^+\CG\to\L_{\nu_i}\dot{(\CG)}_\ul s)_i\right)$; see Remark \ref{RemFormalTorsor}\ref{formal torsor e}. To define the map $\Psi_\CG$, we send the tuple $( \CL_{+,\nu_i},\delta_{\nu_i}:\L_{\nu_i}\CL_{+,\nu_i}\tilde{\to}\L_{\nu_i}\L_{\nu_i}^+\CG )_i$ to the $\FG$-bundle associated with the tuple $(\dot{(\CG_T)}_{\ul s_T},(\CL_{+,\nu_i}, \alpha_i\circ\delta_{\nu_i}:\L_{\nu_i}\CL_{+,\nu_i} \to\L_{\nu_i}\dot{(\CG_T)}_{\ul s_T})_i)$.\\ 
To see that the resulting map $\Psi_\cG$ is formally smooth, first notice that we may reduce to the case where $\L_{\nu_i}\L_{\nu_i}^+\CG$ is trivial. This is because being formally smooth is \'etale local on the source and target. Then using the assumption that $\FG$ is smooth and $\scrH^1(C,\FG)$ is locally noetherian, one can easily argue by lifting criterion for smoothness.

\end{proof}

\subsection{Local Boundedness Conditions}

Here we recall the notion of \emph{local boundedness condition}, introduced in \cite[Definition~4.8]{AH_Local}, and we further explain the relation to the global boundedness condition which we introduced in section \ref{SubsecGLobBound}.\\
Fix an algebraic closure $\BaseOfD\dpl\zeta\dpr^\alg$ of $\BaseOfD\dpl\zeta\dpr$. Since its ring of integers is not complete we prefer to work with finite extensions of discrete valuation rings $R/\BaseOfD\dbl\zeta\dbr$ such that $R\subset\BaseOfD\dpl\zeta\dpr^\alg$. For such a ring $R$ we denote by $\kappa_R$ its residue field, and we let $\Nilp_R$ be the category of $R$-schemes on which $\zeta$ is locally nilpotent. We also set $\wh{\SpaceFl}_{\BP,R}:=\SpaceFl_\BP\whtimes_{\BaseOfD}\Spf R$ and $\wh{\SpaceFl}_\BP:=\wh{\SpaceFl}_{\BP,\BaseOfD\dbl\zeta\dbr}$. Before we can define (local) ``bounds'' we need to make the following observations.

\begin{definition}\label{DefEqClClosedInd}
(a) For a finite extension of discrete valuation rings $\BaseOfD\dbl\zeta\dbr\subset R\subset\BaseOfD\dpl\zeta\dpr^\alg$ we consider closed ind-subschemes $\wh Z_R\subset\wh{\SpaceFl}_{\BP,R}$. We call two closed ind-subschemes $\wh Z_R\subset\wh{\SpaceFl}_{\BP,R}$ and $\wh Z'_{R'}\subset\wh{\SpaceFl}_{\BP,R'}$ \emph{equivalent} if there is a finite extension of discrete valuation rings $\BaseOfD\dbl\zeta\dbr\subset\wt R\subset\BaseOfD\dpl\zeta\dpr^\alg$ containing $R$ and $R'$ such that $\wh Z_R\whtimes_{\Spf R}\Spf\wt R \,=\,\wh Z'_{R'}\whtimes_{\Spf R'}\Spf\wt R$ as closed ind-subschemes of $\wh{\SpaceFl}_{\BP,\wt R}$.

\medskip\noindent
(b) Let $\wh Z=[\wh Z_R]$ be an equivalence class of closed ind-subschemes $\wh Z_R\subset\wh{\SpaceFl}_{\BP,R}$ and let $G_{\wh Z}:=\{\gamma\in\Aut_{\BaseOfD\dbl\zeta\dbr}(\BaseOfD\dpl\zeta\dpr^\alg)\colon \gamma(\wh Z)=\wh Z\,\}$. We define the \emph{ring of definition $R_{\wh Z}$ of $\wh Z$} as the intersection of the fixed field of $G_{\wh Z}$ in $\BaseOfD\dpl\zeta\dpr^\alg$ with all the finite extensions $R\subset\BaseOfD\dpl\zeta\dpr^\alg$ of $\BaseOfD\dbl\zeta\dbr$ over which a representative $\wh Z_R$ of $\wh Z$ exists.
\end{definition}

\begin{definition}\label{DefBDLocal}
\begin{enumerate}
\item \label{DefBDLocal_A}
We define a (local) \emph{bound} to be an equivalence class $\wh Z:=[\wh Z_R]$ of closed ind-subschemes $\wh Z_R\subset\wh{\SpaceFl}_{\BP,R}$, such that all the ind-subschemes $\wh Z_R$ are stable under the left $L^+\BP$-action on $\SpaceFl_\BP$, and the special fibers $Z_R:=\wh Z_R\whtimes_{\Spf R}\Spec\kappa_R$ are quasi-compact subschemes of $\SpaceFl_\BP\whtimes_{\BaseOfD}\Spec\kappa_R$. The ring of definition $R_{\wh Z}$ of $\wh Z$ is called the \emph{reflex ring} of $\wh Z$. Since the Galois descent for closed ind-subschemes of $\SpaceFl_\BP$ is effective, the $Z_R$ arise by base change from a unique closed subscheme $Z\subset\SpaceFl_\BP\whtimes_\BaseOfD\kappa_{R_{\wh Z}}$. We call $Z$ the \emph{special fiber} of the bound $\wh Z$. It is a projective scheme over $\kappa_{R_{\wh Z}}$ by~\cite[Remark 4.3]{AH_Local} and \cite[Lemma~5.4]{H-V}, which implies that every morphism from a quasi-compact scheme to an ind-projective ind-scheme factors through a projective subscheme.
\item \label{DefBDLocal_B}
Let $\wh Z$ be a bound with reflex ring $R_{\wh Z}$. Let $\CL_+$ and $\CL_+'$ be $L^+\BP$-torsors over a scheme $S$ in $\Nilp_{R_{\wh Z}}$ and let $\delta\colon \CL\isoto\CL'$ be an isomorphism of the associated $L\genericG$-torsors. We consider an \'etale covering $S'\to S$ over which trivializations $\alpha\colon\CL_+\isoto(L^+\BP)_{S'}$ and $\alpha'\colon\CL'_+\isoto(L^+\BP)_{S'}$ exist. Then the automorphism $\alpha'\circ\delta\circ\alpha^{-1}$ of $(L\genericG)_{S'}$ corresponds to a morphism $S'\to L\genericG\whtimes_\BaseOfD\Spf R_{\wh Z}$. We say that $\delta$ is \emph{bounded by $\wh Z$} if for any such trivialization and for all finite extensions $R$ of $\BaseOfD\dbl\zeta\dbr$ over which a representative $\wh Z_R$ of $\wh Z$ exists the induced morphism $S'\whtimes_{R_{\wh Z}}\Spf R\to L\genericG\whtimes_\BaseOfD\Spf R\to \wh{\SpaceFl}_{\BP,R}$ factors through $\wh Z_R$. Furthermore we say that a local $\BP$-shtuka $(\CL_+, \tauLoc)$ is \emph{bounded by $\wh Z$} if the isomorphism $\tauLoc$ is bounded by $\wh Z$. Assume that $\wh Z=\CS(\omega)\whtimes_\BaseOfD\Spf \BaseOfD\dbl\zeta\dbr$ for a \emph{Schubert variety} $\CS(\omega)\subseteq \CF\ell_\BP$ , with $\omega\in \wt{W}$; see \cite{PR2}. Then we say that $\delta$ is \emph{bounded by $\omega$}. 

\item\label{DefBDLocal_C}
Fix an $n$-tuple $\ul\nu=(\nu_i)$ of places on the curve $C$ with $\nu_i\ne\nu_j$ for $i\ne j$. Let $\wh Z_\ul\nu:=(\wh Z_{\nu_i})_i$ be an $n$-tuple of bounds in the above sense and set $R_{\wh Z_\ul\nu}:=R_{\wh Z_{\nu_1}}\hat{\otimes}_{\BF_q}\dots \hat{\otimes}_{\BF_q} R_{\wh Z_{\nu_n}}$. We say that a tuple $(\CG,\CG',\ul s,\phi)$ in $Hecke_n(C,\FG)^\ul\nu\times_{\wh A_\ul\nu}\Spf R_{\wh Z_\ul\nu}$ is bounded by $\wh Z_\ul\nu$ if for each $i$ the associated isomorphism $\wh\phi_{\nu_i}:=\L_{\nu_i}(\phi_i):\L_{\nu_i}\CG'\to \L_{\nu_i}\CG$ is bounded by $\wh Z_{\nu_i}$ in the above sense. We denote the resulting formal stack by $Hecke_n^{\wh Z_\ul\nu}(C,\FG)^\ul\nu$, and sometimes we abbreviate this notation by $Hecke_n^{\wh Z_\ul\nu}$.  This accordingly defines boundedness condition on global $\FG$-shtukas. Equivalently, one says that a global $\FG$-shtuka $\ul\CG$ in $\nabla_n \scrH^1 (C,\FG)^{\ul\nu}(S)$ is bounded by $\wh Z_\ul\nu$ if for each $i$ the associated local $\BP_{\nu_i}$-shtuka $\wh{\Gamma}_{\nu_i}(\ul\CG)$ under the global-local functor $\wh{\Gamma}_{\nu_i}(-)$ is bounded by $\wh Z_{\nu_i}$. We denote by $\nabla_n^{\wh Z_\ul\nu}\scrH^1(C,\FG)^\ul\nu$ the formal substack obtained by imposing the boundedness condition $\wh Z_\ul\nu$. 
\item\label{DefBDLocal_D}
Assume that the bound $\wh Z_\ul\nu$ comes from a tuple of affine Schubert varieties $\CS(\ul\omega):=(\CS(\omega_i))_i$, where $\ul\omega:=(\omega_i)_i\in\prod_{i=1}^n\wt W_i$. Here $\wt W_i$ denotes the Iwahori-Weyl group corresponding to $P_{\nu_i}$; see Definition \ref{DefIwahori-Weyl}. Then we use the notation $Hecke_n^\ul\omega (C,\FG)^\ul\nu$ (resp. $\nabla_n^{\ul\omega}\scrH^1(C,\FG)^\ul\nu$) for the corresponding moduli stack obtained by imposing the bound $\wh Z_\ul\nu$.

\end{enumerate}
\end{definition}

\begin{proposition}

Fix an $n$-tuple $\ul\nu=(\nu_i)_i$ of pairwise distinct places on the curve $C$. Let $\CZ$ be a global bound in the sense of Definition \ref{DefAltGlobalBC}. Furthermore, let $\ul\nu'=(\nu_i')_i$ be an $n$-tuple of places on the reflex field $Q_\CZ$, such that $\nu_i'$ lies over $\nu_i$. Then one can associate an $n$-tuple $(\wh Z_{\nu_i})$ of local bounds to the global bound $\CZ$.

\end{proposition}

\begin{proof}
Let $Z:=Z_K$ be a representative of $\CZ$ over a field $K$ which has minimal inseparable degree over $Q$.
Let $Z_i$ denote the image of $Z$ under the morphism $GR_n\times_{C^n} \wt C_K^n\to GR_1\times_{C} \wt C_K$ given by $(\ul s:=(s_i)_i, \wh \CG, \dot{\epsilon})\mapsto(s_i,\wh \CG|_{\BD(\Gamma_{s_i})}, \dot{\epsilon}|_{\dot{\BD}(\Gamma_{s_i})})$; see Remark \ref{RemarkBeauvilleLaszlo}. Note that $Z_i$ is a closed \forget{\comment{here we require GR to be projective}}subscheme of $GR_1\times_C \wt C$ which is stable under the action of the global loop group $\FL_1^+\FG$. Let $\wh Z_i^j$ denote the scheme defined by the following Cartesian diagram

$$
\xymatrix@!0{ &&Z_i\ar[rr]&&& GR_1\times_C \wt C_K \ar[rrrr]\ar[dd] &&&& GR_1 \ar[dd] \\ 
\coprod\wh{Z}_i^j \ar[rr]\ar[urr]&&&~~\coprod_j \CF \ell_{{\BP_{\nu_i}},\wt A_{\nu_i^j}}\ar[urr]\ar[rrrr]\ar[dd] &&&& \cF\ell_{\BP_{\nu_i},\wh A_{\nu_i}} \ar[urr]\ar[dd] \\ 
&&&&& \wt C_K \ar[rrrr] &&&& C \\ 
&&&\coprod_j\Spf \wt A_{\nu_i^j} \ar[rrrr]\ar[urr] &&&& \Spf \wh A_{\nu_i} \ar[urr] } 
$$

Now we choose an element $\wh{Z}_{K,i}:=\wh Z_i^j$ from the set $\{\wh{Z}_i^j\}_j$ such that the corresponding place $\nu_i^j$ lies above $\nu_i'$. We assign to a global bound $\CZ$, the tuple $([\wh Z_{K,i}])_i$ of local bounds. Note that since the extension $Q_\CZ\subset K$ is separable, see Remark~\ref{RemAltReflexRing}, this assignment is independent of the choice of $\wh{Z}_{K,i}$. The fact that the assignment is independent of the choice of the representative representative $Z_K$ of $\CZ$ is obvious.

\end{proof}

\subsection{The Local Model Theorem For $\nabla_n^{H,\wh{Z}_\nu}\scrH^1(C,\FG)^\nu$}\label{SubsectionLMTII}

\noindent
Before constructing the local model diagram for the moduli stacks of global $\FG$-shtukas, let us first treat the case of Hecke stacks.

\begin{definition}\label{Hecketilda}
Let $D\subset C$ be a finite subscheme, disjoint from $\ul\nu$. We denote by $\wt{Hecke}_{n,D}(C,\FG)^\ul\nu$ the formal stack whose $S$-points parametrize tuples $\left(((\CG,\psi),(\CG',\psi'),\ul s,\tau_i),(\epsilon_i)_i\right)$, consisting of 
\begin{enumerate}
\item[i)] $((\CG,\psi),(\CG',\psi'),\ul s,\tau)$ in $Hecke_{n,D}(C,\FG)^\ul\nu(S)$ and
\item[ii)] trivializations $\epsilon_i:\L_{\nu_i}^+(\CG')\tilde{\to} L^+\BP_{\nu_i,S}$ of the associated  $L^+\BP_{\nu_i}$-torsors $\L_{\nu_i}^+(\CG')$; see Remark~\ref{RemFormalTorsor}\ref{formal torsor c}.
\end{enumerate}
\noindent

Moreover, for a bound $\wh Z_{\ul\nu}:=(\wh{Z}_{\nu_i})_i$, we denote by $\wt{Hecke}_{n,D}^{\wh Z_{\ul\nu}}(C,\FG)^\ul\nu$ the formal substack obtained by imposing the bound $\wh Z_{\ul\nu}$ to the isomorphism $\tau$. Since we fixed the curve $C$ and the group $\FG$, and moreover, the number of characteristic places is implicit in $\ul\nu$, we sometimes drop them from our notation and we write $\wt{Hecke}_D^\ul\nu$ (resp. $\wt{Hecke}_{D}^{\wh Z_{\ul\nu}}$) instead of $\wt{Hecke}_{n,D}(C,\FG)^\ul\nu$ (resp. $\wt{Hecke}_{n,D}^{\wh Z_{\ul\nu}}(C,\FG)^\ul\nu$). We further drop the subscript $D$ when it is empty.

\end{definition}
\bigskip

\noindent
Let $\wh{Z}_{\nu_i,R_{\nu_i}}$ be a representative of $\wh Z_{\nu_i}$ over $R_{\nu_i}$. Set $R_{\wh Z_\ul \nu} :=R_{\wh Z_{\nu_1}}\hat{\otimes}_{\BF_q}\dots \hat{\otimes}_{\BF_q} R_{\wh Z_{\nu_n}}$ and $R_\ul \nu :=R_{\nu_1}\hat{\otimes}_{\BF_q}\dots \hat{\otimes}_{\BF_q}R_{\nu_n}$ and let $Hecke_{D,R_\ul\nu}^{\wh Z_{\ul\nu}}$ (resp. $\wt{Hecke}_{D,R_{\ul\nu}}^{\wh Z_\ul\nu}$) denote the base change $Hecke_D^{\wh Z_{\ul\nu}}\times_{R_{\wh Z_\ul\nu}}R_\ul\nu$ (resp. $\wt{Hecke}_D^{\wh Z_{\ul\nu}} \times_{R_{{\wh Z_\ul\nu}}}R_\ul\nu$).

\begin{proposition}\label{PropLocalModelHecke}
There is a roof of morphisms 
\begin{equation}\label{HeckeRoof}
\xygraph{
!{<0cm,0cm>;<1cm,0cm>:<0cm,1cm>::}
!{(0,0) }*+{\wt{Hecke}_{D,R_{\ul\nu}}^{\wh Z_\ul\nu}}="a"
!{(-1.5,-1.5) }*+{Hecke_{D,R_\ul\nu}^{\wh Z_{\ul\nu}}}="b"
!{(1.5,-1.5) }*+{\prod_i \wh Z_{\nu_i,R_{\nu_i}}.}="c"
"a":^{\pi}"b" "a":_{f}"c"
}
\end{equation}  
\noindent
Furthermore, in the above roof, the formal stack $\wt{Hecke}_{D,R_{\ul\nu}}^{\wh Z_\ul\nu}$ is an $\prod_i L^+\BP_{\nu_i}$-torsor over $Hecke_{D,R_\ul\nu}^{\wh Z_{\ul\nu}}$ under the projection $\pi$. Moreover for a geometric point $y$ of $Hecke_{D,R_\ul\nu}^{\wh Z_{\ul\nu}}$, the $\prod_i L^+\BP_{\nu_i}$-torsor $\pi:\wt{Hecke}_{D,R_{{\wh Z_\ul\nu}}}^{\wh Z_{\ul\nu}}\to Hecke_{D,R_\ul\nu}^{\wh Z_{\ul\nu}}$ admits a section $s$, over an \'etale neighborhood of $y$, such that the composition $f\circ s$ is formally smooth\forget{ of relative dimension $d:=\dim \scrH^1(C,\FG)$}. 
\end{proposition}

\begin{proof}  

Forgetting the trivialization $\epsilon_i$, it is clear that the formal stack $\wt{Hecke}_{D,R_{\ul\nu}}^{\wh Z_\ul\nu}$ is an $\prod_i L^+\BP_{\nu_i}$-torsor over $Hecke_{D,R_{\ul\nu}}^{\wh Z_\ul\nu}$.\\
\noindent
 Recall that $\wh{\SpaceFl}_{\BP}$ represents the functor (\ref{EqRZFunctor}). Consider a tuple $((\CG,\CG',\ul s,\tau), (\epsilon_i)_i)$ in $\wt{Hecke}^\ul\nu$ and let $\alpha_i$ (resp. $\alpha_i'$) denote the canonical isomorphism $\L_{\nu_i} \L_{\nu_i}^+\CG\tilde{\to}\L_{\nu_i}\dot{(\CG)}_\ul s$ (resp. $\L_{\nu_i} \L_{\nu_i}^+\CG'\tilde{\to}\L_{\nu_i}\dot{(\CG')}_\ul s$). We define the morphism $\wt{Hecke}^\ul\nu \to \prod_{\nu_i} \wh{\SpaceFl}_{\BP_{\nu_i}},$
which assigns $(\L_{\nu_i}^+\CG ,\phi_i:=\L_{\nu_i}(\epsilon_i)\circ\alpha_i'^{-1}\circ \L_{\nu_i}(\tau)^{-1}\circ \alpha_i)$ to the tuple $((\CG,\CG',\ul s,\tau), (\epsilon_i)_i)$. This morphism\forget{factors through $\prod_i \wh Z_{\nu_i}$ and} induces the following map
\begin{eqnarray}\label{LocalModelRoofRightArrow}
f:\wt{Hecke}_{D,R_\ul\nu}^{\wh Z_{\ul\nu}} \longto \prod_i \wh{Z}_{\nu_i,R_{\nu_i}}.
\end{eqnarray}
\noindent
Let $\CA:=\CA_y$ be the stalk of the structure sheaf of $Hecke^{\wh Z_{\ul\nu}}$ at the geometric point $y$. 
As $\CA$ is strictly henselian, we may fix a trivialization $\epsilon_i^\CA$ of the restriction of the universal $L^+\BP_{\nu_i}$-torsor over $\scrH^1(\Spec \BF_\nu, L^+\BP_{\nu_i})$ to $\Spec \cA$; see Remark~\ref{RemFormalTorsor}\ref{formal torsor a}. This induces the section $s$.

\noindent
Notice that as $Hecke_{D,R_\ul\nu}^{\wh Z_{\ul\nu}}$ is a torsor over $Hecke_{R_\ul\nu}^{\wh Z_{\ul\nu}}$ for the smooth group scheme $\FG_D=\Res_{D/k}\FG$, we may ignore the $D$-level structure. The morphism $f:\wt{Hecke}_{R_\ul\nu }^{\wh Z_{\ul\nu}}\longto\prod_i \wh Z_{\nu_i,R_{\nu_i}}$ factors through $f_1:\wt{Hecke}_{R_\ul\nu }^{\wh Z_{\ul\nu}} \longto  \scrH^1(C,\FG)\times_{\BF_q}\prod_i \wh Z_{\nu_i,R_{\nu_i}}$ 
followed by the projection to the second factor. \forget{As any formal $\hat{\BP}_\nu$-torsor over $\BD(\Gamma_\ul s)$ extends to a $\FG$-bundle over $C_S$, and $\scrH^1(C,\FG)$ is smooth \cite[Theorem~2.4]{AH_Global}, it suffices to show that $f_1\circ s$ is formally smooth over its image.}Consider a closed immersion $\ol S \hookrightarrow S$, defined by a nilpotent sheaf of ideal $\CI$. Let $(\ol \CG, \ol \CG',(\ol s_i), \ol\tau)$ be an object of $Hecke^{\wh Z_{\ul\nu}}(\ol S)$ and assume that it maps to the tuple $\left(\ol\CG,(\L_{\nu_i}^+\ol\CG,\ol \phi_i:\L_{\nu_i}\ol\CG\tilde{\to}LP_{\nu_i,\ol S})_i\right)$ (resp. $(\L_{\nu_i}^+\ol\CG,\ol \phi_i)_i$) via $f_1\circ s$ (resp. $f \circ s$). Let $(\CL_{+,i},\phi_i:\CL_{i}\tilde{\to}LP_{\nu_i,  S})_i$ be a lift of $(\L_{\nu_i}^+\ol\CG,\ol \phi_i)_i$ over $S$. As any formal $\hat{\BP}_\nu$-torsor over $\BD(\Gamma_\ul s)$ extends to a $\FG$-bundle over $C_S$, and $\scrH^1(C,\FG)$ is smooth \cite[Theorem~2.4]{AH_Global}, one may take a lift $\left(\CG,(\CL_{+,i}=\L_{\nu_i}^+\CG,\phi_i:\CL_{i}\tilde{\to}LP_{\nu_i,  S})_i\right)$ of $\left(\ol\CG,(\L_{\nu_i}^+\ol\CG,\ol\phi_i)_i \right)$ over $S$, which maps to $(\CL_{+,i},\phi_i:\CL_{i}\tilde{\to}LP_{\nu_i,  S})_i$. Again by smoothness of the algebraic stack $\scrH^1(C,\FG)$, we may choose a $\FG$-bundle $\CG'$ which lifts $\ol \CG'$ over $S$. We let $\epsilon_{i,S}'$ denote the trivialization of $\L_{\nu_i}^+\CG'$ induced by $\epsilon_i^\CA$.

\noindent
Consider the isomorphism $\delta_{\nu_i}:=\phi_i^{-1}\circ \L(\epsilon_{i,S}'):\L_{\nu_i}\L_{\nu_i}^+\CG'\tilde{\to}\L_{\nu_i}\L_{\nu_i}^+\CG$ and let $\CG''$ denote the image of $S$-point $(\L_{\nu_i}^+\CG',\delta_{\nu_i})_i$ under the uniformization map
$$
\Psi_\CG:\prod_{\nu_i} \CM_{\nu_i}(\L_{\nu_i}\CG) \to \scrH^1(C,\FG)|_S,
$$
see Lemma \ref{LemUnifMap}. Note that by construction there is an isomorphism $\tau:\dot{(\cG'')}_\ul s \to \dot{(\cG)}_\ul s$. The $S$-point $(\CG,\CG'', (s_i), \tau)$ of $Hecke_n(C,\FG)^\ul \nu$ provides the desired lift of the $\ol S$-point $(\ol \CG, \ol \CG',(\ol s_i), \ol\tau)$ in the following sense. Namely, there is an isomorphism between $\ol\CG'=(\dot{(\ol\CG')}_\ul s, (\L_{\nu_i}^+\ol\CG',\ol\alpha_i':\L_{\nu_i}\L_{\nu_i}^+\ol\CG'\to \L_{\nu_i}\dot{(\ol\CG')}_\ul s)_i)$ and the pull back $\ol\CG''=\left(\dot{(\ol\CG)}_\ul s, (\L_{\nu_i}^+\ol\CG', \alpha_i'': \L_{\nu_i} \L_{\nu_i}^+\ol\CG'\to \L_{\nu_i}\dot{(\ol\CG)}_\ul s)_i\right)$ of $\CG''$ over $\ol S$, which is given by $\ol\tau$ on the first factor and by identity and the following commutative diagram
$$
\CD
\L_{\nu_i} \L_{\nu_i}^+\ol\CG'@>{\ol\alpha_i'}>>\L_{\nu_i}\dot{(\ol\CG')_\ul s}\\
@| @VV{\L_{\nu_i}(\ol\tau})V\\
\L_{\nu_i} \L_{\nu_i}^+\ol\CG'@>\ol\alpha_i''>>\L_{\nu_i}\dot{(\ol\CG)}_\ul s\;.
\endCD
$$

\noindent
on the second factor. To justify the commutativity of the above diagram notice that 
$$
\ol \alpha_i''=\ol\alpha_i\circ \ol\phi_i^{-1} \circ \L(\epsilon_{i,\ol S}')=\ol\alpha_i\circ \left(\L(\epsilon_{i,\ol S}') \circ \ol\alpha_i'^{-1}\circ \L_{\nu_i}(\ol\tau)^{-1}\circ \ol\alpha_i\right)^{-1} \circ \L(\epsilon_{i,\ol S}')= \L_{\nu_i}(\ol\tau)\circ\ol\alpha_i' 
$$
Regarding the definition of the uniformization map, it can be easily seen that $(\CG,\CG'', \ul s, \tau)$ maps to $(\CL_{+,i},\phi_i:\CL_{i}\tilde{\to}LP_{\nu_i,  S})_i$ under $f\circ s$.
\forget{
It remains to see that $f\circ s$ maps $(\CG,\CG'', \ul s, \tau)$ to $(\CL_{+,i},\phi_i:\CL_{i}\tilde{\to}LP_{\nu_i,  S})_i$
$$
f\circ s ((\CG,\CG'', \ul s, \tau))=f((\CG,\CG'', \ul s, \tau)\times \epsilon_i:\L_{\nu_i}^+\CG'\to L^+\BP_{\nu_i})=(\L_{\nu_i}^+\cG, \L_{\nu_i}(\epsilon_i)\circ \alpha_{\CG''}\circ \L_{\nu_i}(\tau)^{-1}\circ \alpha_\CG))
= (\L_{\nu_i}^+\cG, \phi_i)
$$
}

\end{proof}

Using Tannakian formalism, we can equip the moduli stack $\nabla\scrH^1(C,\FG)^\ul\nu$ of global $\FG$-shtukas with $H$-level structure, for a compact open subgroup $H\subseteq \FG(\BA_Q^{\ul \nu})$; details are explained in \cite[Chapter~6]{AH_Global}. Here we briefly recall the definition.
  
\begin{definition}[$H$-level structure]\label{DefLevelStr}
Assume that $S\in \Nilp_{\wh A_\ul\nu}$ is connected and fix a geometric point $\ol s$ of $S$. Let $\pi_1(S,\bar{s})$ denote the algebraic fundamental group of $S$. 
\begin{enumerate}
\item
The rational Tate functor constructed in \cite[Chapter~6]{AH_Global}
\begin{eqnarray*}
\check{\CV}_{-}\colon \nabla_n\scrH^1(C,\FG)^{\ul\nu}(S) \;&\longto &\; Funct^\otimes (\Rep_{\BA^{\ul\nu}}\FG\,,\,\FM od_{\BA_Q^{\ul\nu}[\pi_1(S,\bar{s})]})\,\\
\ul \CG &\mapsto & \check{\CV}_{\ul\CG}:\rho\to\invlim[D\subset C']\rho_\ast(\ul\CG|_{D_\ol s})^\tau\otimes_{\BA^\ul\nu}\BA_Q^\ul\nu
\end{eqnarray*}
\noindent
assigns to a global $\FG$-shtuka over $S$, a tensor functor from the category $\Rep_{\BA^{\ul \nu}} \FG$ of adelic representation of $\FG$ to the category $\FM od_{\BA_Q^{\ul \nu}[\pi_1(S,\bar{s})]}$ of $\BA_Q^{\ul \nu}[\pi_1(S,\bar{s})]$-modules. The limit is taken over all finite subschemes $D$ of $C':=C\setminus \{\nu_1,\dots,\nu_n\}$.
\item 
For a global $\FG$-shtuka $\ul\CG$ over $S$ let us consider the sets of isomorphisms of tensor functors $\Isom^{\otimes}(\check{\CV}_{\ul{\CG}},\omega^\circ)$, where $\omega^\circ\colon \Rep_{\BA^{\ul \nu}}\FG \to \FM od_{\BA_Q^{\ul \nu}}$ denote the neutral fiber functor. The set $\Isom^{\otimes}(\check{\CV}_{\ul{\CG}},\omega^\circ)$ admits an action of $\FG(\BA_Q^{\ul \nu})\times\pi_1(S,\bar{s})$ where $\FG(\BA_Q^{\ul \nu})$ acts through $\omega^\circ$ by Tannakian formalism and $\pi_1(S,\bar{s})$ acts through $\check{\CV}_{\ul\CG}$. For a compact open subgroup $H\subseteq \FG(\BA_Q^{\ul \nu})$ we define a \emph{rational $H$-level structure} $\bar\gamma$ on a global $\FG$-shtuka $\ul \CG$ over $S\in\Nilp_{\wh A_\ul\nu}$ as a $\pi_1(S,\bar{s})$-invariant $H$-orbit $\bar\gamma=H\gamma$ in $\Isom^{\otimes}(\check{\CV}_{\ul{\CG}},\omega^\circ)$.
\item 
We denote by $\nabla_n^{H,\wh Z_{\ul\nu}}\scrH^1(C,\FG)^{\ul \nu}$ the category fibered in groupoids, whose category of $S$-valued points $\nabla_n^{H,\wh Z_{\ul\nu}}\scrH^1(C,\FG)^{\ul \nu}(S)$ has tuples $(\ul \CG,\bar\gamma)$, consisting of a global $\FG$-shtuka $\ul\CG$ in $\nabla_n^{\wh Z_{\ul\nu}} \scrH^1(C,\FG)^{\ul\nu}(S)$ together with a rational $H$-level structure $\bar\gamma$, as its objects. The morphisms are quasi-isogenies of global $\FG$-shtukas that are isomorphisms at the characteristic places $\nu_i$ and are compatible with the $H$-level structures.  

\end{enumerate}

\end{definition}
 
 \noindent
The above definition of level structure generalizes the initial Definition~\ref{Global Sht}, according to the following proposition. 

\begin{proposition}\label{PropH_DL-Str}
We have the following statements
\begin{enumerate} 
\item\label{PropH_DL-Str a)}
\forget{\comment{Let $\CZ$ be a global bounded and let $\wh Z_\ul\nu=(Z_{\nu_i})$ denote the associated tuple of local bounds.}}For a finite subscheme $D\subset C$, disjoint from $\ul\nu$, there is a canonical isomorphism $\nabla_n\scrH_D^1 (C,\FG)^{\ul\nu}\es \isoto \es\nabla_n^{H_D}\scrH^1(C,\FG)^{\ul\nu}$\forget{\comment{$\nabla_n^\CZ\scrH_D^1 (C,\FG)\otimes_{C^n}\Spf A_{\ul\nu}\es \isoto \es\nabla_n^{H_D, \wh Z_\ul\nu}\scrH^1(C,\FG)^{\ul\nu}$}} of formal stacks. Here $H_D$ denotes the compact open subgroup $\ker\bigl(\FG(\BA^{\ul\nu})\to\FG(\CO_D)\bigr)$ of $\FG(\BA^{\ul\nu}_Q)$. 
\item\label{PropH_DL-Str b)}
For any compact open subgroup $H \subseteq \FG(\BA_Q^\ul\nu)$ the stack $\nabla_n^H\scrH^1 (C,\FG)^{\ul\nu}$ is an ind-algebraic stack, ind-separated and locally of ind-finite type over $\Spf A_{\ul\nu}$. The forgetful morphism $\nabla_n^H\scrH^1 (C,\FG)^{\ul\nu} \to \nabla_n\scrH^1 (C,\FG)^{\ul\nu}$ is finite \'etale. 
\end{enumerate}
\end{proposition}

\begin{proof}
For part \ref{PropH_DL-Str a)} see \cite[Theorem~6.4]{AH_Global}. Part \ref{PropH_DL-Str b)} follows from \ref{PropH_DL-Str a)} and  Theorem \ref{ThmnHisArtin}.
\end{proof}

\begin{definition}\label{DefNablaOmegaHTilde}
We define the formal stack $\nabla_n^{\wh Z_{\ul\nu}} \wt{{\scrH^1}(C, \FG)}^{\ul\nu}$ by the following pull back diagram

$$
\CD
\nabla_n^{\wh Z_{\ul\nu}} \wt{{\scrH^1}(C, \FG)}^{\ul\nu}@>>>\wt{Hecke}_n^{\wh Z_{\ul\nu}}\\
@VVV @VVV\\
\nabla_n^{\wh Z_{\ul\nu}} {\scrH^1}(C, \FG)^{\ul\nu}@>>>Hecke_n^{\wh Z_{\ul\nu}}\;.
\endCD
$$

\noindent
More explicitly the $S$-points of the formal stack $\nabla_n^{\wh Z_{\ul\nu}} \wt{{\scrH^1}(C, \FG)}^{\ul\nu}$ parametrizes the tuples $(\ul\CG,(\epsilon_i)_i)$ consisting of 
\begin{enumerate}
\item\label{DefNablaOmegaHTildeI}
a $\FG$-shtuka $\ul\CG$ in $\nabla_n^{\wh Z_{\ul\nu}} \scrH^1(C, \FG)^{\ul\nu}(S)$ and
\item\label{DefNablaOmegaHTildeII}
trivializations $\epsilon_i:\hat{\sigma}_S\CL_{+,i}\isoto L^+\BP_{i,S}$, where $(\CL_{+,i},\hat{\tau}_i):=\wh{\Gamma}_{\nu_i}(\ul\CG)$.

\end{enumerate}
\noindent
Furthermore we use the notation $\nabla_n^{H, \wh Z_{\ul\nu}} \wt{{\scrH^1}(C, \FG)}^{\ul\nu}$ when the $\FG$-shtukas in \ref{DefNablaOmegaHTildeI} are additionally equipped with $H$-level structure. Since we fix the curve $C$, the group $\FG$ and the characteristic places $\ul\nu$, for the sake of simplicity we sometimes drop them from our notation and will write $\nabla_n^{H, \wh Z_{\ul\nu}} {{\scrH^1}}$ and $\nabla_n^{H, \wh Z_{\ul\nu}} {\wt{\scrH^1}}$ to denote the corresponding formal stacks.
\end{definition}

\begin{theorem}\label{ThmRapoportZinkLocalModel}
Keep the above notation. Consider the following roof 

\begin{equation}\label{nablaHRoof} 
\xygraph{
!{<0cm,0cm>;<1cm,0cm>:<0cm,1cm>::}
!{(0,0) }*+{\nabla_n^{H, \wh Z_{\ul\nu}}\wt{\scrH_{R_\ul\nu}^1}}="a"
!{(-1.5,-1.5) }*+{\nabla_n^{H, \wh Z_{\ul\nu}}\scrH_{R_\ul\nu}^1}="b"
!{(1.5,-1.5) }*+{\prod_i \wh Z_{\nu_i,R_{\nu_i}},}="c"
"a":^{\pi'}"b" "a":_{f'}"c"
}  
\end{equation}

\noindent
induced from \ref{HeckeRoof}. Let $y$ be a geometric point of $\nabla_n^{H, \wh Z_{\ul\nu}}\scrH_{R_\ul\nu}^1$. The $\prod_i L^+\BP_{\nu_i}$-torsor $\pi': \nabla_n^{H, \wh Z_{\ul\nu}}\wt{\scrH_{R_\ul\nu}^1}\to\nabla_n^{H, \wh Z_{\ul\nu}}\scrH_{R_\ul\nu}^1$ admits a section $s'$, locally over an \'etale neighborhood of $y$, such that the composition $f'\circ s'$ is formally \'etale. 

\end{theorem}

\begin{proof}

According to Proposition \ref{PropH_DL-Str} we may forget the $H$-level structure. Let $\CA':=\CA_y'$ be the henselian local ring at the geometric point $y\in\nabla_n^{\wh Z_{\ul\nu}}\scrH_{R_\ul\nu}^1$. The restriction of the section $s$ obtained in the course of the proof of Proposition \ref{PropLocalModelHecke} to $\CU':=\Spec \CA'$ provides a section $s'$.


Now we check that the morphism $f' \circ s'$ is formally \'etale. Let $\ol S$ be a closed subscheme of $S$, defined by a nilpotent sheaf of ideal $\CI$. Take an $\ol S$-valued point $\ol{\ul\CG}=(\ol \CG, (s_i), \ol\tau)$ of $\CA'$ and let $(\ol{\ul\CL}_i)_i=\left((\ol\CL_{+,i},\hat{\ol\tau}_i)\right)_i$ denote the associated tuple of local $\BP_{\nu_i}$-shtukas under the global-local functor.  Let $(\ol \CL_{+,i}, \ol\phi_i:\ol \CL_i\isoto  L\BP_{\nu_i,\ol S})$ denote the image of $\ol{\ul\CG}=(\ol\CG, (s_i), \ol \tau)$ in $\prod_i \wh Z_{\nu_i}(\ol S)$ under the morphism $f' \circ s'$, and let $(\CL_{+,i}, \phi_i:\CL_i\isoto  L\BP_{\nu_i,S})$ be an infinitesimal deformation of $(\ol \CL_{+,i}, \phi_i:\ol \CL_i\isoto  L\BP_{\nu_i,\ol S}) \in \prod_i \wh Z_{\nu_i}(\ol S)$ over $S$. We consider the infinitesimal deformation $(\ul\CL_i)_i:=\left((\CL_{+,i},\hat{\tau}_i:=\phi_i^{-1}\circ \L_{\nu_i}(\epsilon_{i,S}))_i\right)_i$ of the $n$-tuple $(\ol{\ul\CL}_i)_i$ of local $\BP_{\nu_i}$-shtukas over $S$. Here by $\epsilon_{i,S}$ denote the trivialization of $\sigma_S^\ast\CL_+$ induced by $\epsilon_i^\CA$; see proof of the Proposition \ref{PropLocalModelHecke}. By the analog of Serre-Tate, theorem \ref{Serre-Tate}, we have $$Defo_S(\bar{\ul{\cG}})\cong \prod_i Defo_S(\ul{\bar\CL}_i).$$
Consequently, the $n$-tuple $(\ul\CL_i)_i$ provides a unique infinitesimal deformation $\ul\CG=(\CG, \ul s, \tau)$ of $\ol{\ul\CG}=(\ol \CG, \ul s, \ol\tau)$ over $S$, which by construction maps to $(\CL_{+,i}, \phi_i)$ under $f'\circ s'$. The fact that $\ul\CG$ is bounded by $\wh Z_{\ul\nu}$ is tautological. \\

\end{proof}

\begin{remark}
Analyzing the proof of Theorem \ref{ThmLocalModelI} and also Proposition~\ref{PropLocalModelHecke}, we immediately see that the assumption of smoothness of $\FG$ over $C$ can not be weakened.
\end{remark}

\subsection{Some Applications}
In this section we assume that $\FG$ is a parahoric Bruhat-Tits group scheme (i.e. a smooth affine group scheme with connected fibers and reductive generic fiber) over $C$.\\

Let $\wh Z_\ul\nu$ be a bound and $\wh{Z}_{\nu_i,R_{\nu_i}}$ be a representative of the bound $\wh Z_{\nu_i}$ over $R_{\nu_i}$. 
Let $K_{\ul\nu}=\prod_i K_{\nu_i}\subset \prod_i L^+\BP_{\nu_i}$ denote the normal subgroup which acts trivially on $\wh Z_{\ul\nu}$. The diagram \ref{nablaHRoof} (resp. \ref{HeckeRoof}) induces the following Cartesian diagram
\[ 
\xygraph{
!{<0cm,0cm>;<1cm,0cm>:<0cm,1cm>::}
!{(0,0) }*+{\left( \nabla_n^{H, \wh Z_{\ul\nu}}\wt{\scrH_{R_\ul\nu}^1}\right)^{K_{\ul\nu}}}="a"
!{(-2,-2) }*+{\nabla_n^{H, \wh Z_{\ul\nu}}\scrH_{R_\ul\nu}^1}="b"
!{(2,-2) }*+{\prod_i [K_{\nu_i}\backslash \wh Z_{\nu_i,R_{\nu_i}}]}="c"
!{(0,-4) }*+{\prod_i [L^+\BP_{\nu_i}\backslash\wh Z_{\nu_i,R_{\nu_i}}]}="d"
"a":^{\pi^{K_{\ul\nu}}}"b" "a":_{f^{K_{\ul\nu}}}"c" "c":"d" "b":"d"
}  
\]

(resp.
\[ 
\xygraph{
!{<0cm,0cm>;<1cm,0cm>:<0cm,1cm>::}
!{(0,0) }*+{\left( \wt{Hecke}_{D,R_\ul\nu}^{\wh Z_{\ul\nu}}\right)^{K_{\ul\nu}}}="a"
!{(-2,-2) }*+{Hecke_{D,R_\ul\nu}^{\wh Z_{\ul\nu}}}="b"
!{(2,-2) }*+{\prod_i [K_{\nu_i}\backslash \wh Z_{\nu_i,R_{\nu_i}}]~~~~)}="c"
!{(0,-4) }*+{\prod_i [L^+\BP_{\nu_i}\backslash\wh Z_{\nu_i,R_{\nu_i}}]}="d"
"a":^{\pi^{K_{\ul\nu}}}"b" "a":_{f^{K_{\ul\nu}}}"c" "c":"d" "b":"d"
}  
\]

\noindent
of algebraic stacks. Here $[L^+\BP_{\nu_i}\backslash\wh Z_{\nu_i,R_{\nu_i}}]$ and $[K_{\nu_i}\backslash \wh Z_{\nu_i,R_{\nu_i}}]$ denote the quotient stacks and $\left(\nabla_n^{H,\wh Z_{\ul\nu}}\wt{\scrH_{R_\ul\nu}^1}\right)^{K_{\ul\nu}}$ (resp. $\left( \wt{Hecke}_{D,R_\ul\nu}^{\wh Z_{\ul\nu}}\right)^{K_{\ul\nu}}$), denote the formal algebraic stack whose $S$-points parametrizes $K_\ul\nu$-orbits of the $S$-points $(\ul\CG,(\epsilon_i)_i)$ (resp. $\left((\CG,\CG',\ul s,\tau),(\epsilon_i)_i\right)$) of $\nabla_n^{H, \wh Z_{\ul\nu}}\wt{\scrH_{R_\ul\nu}^1}(S)$ (resp. $\wt{Hecke}_{D, R_\ul\nu}^{\wh Z_{\ul\nu}}(S)$). This stack is a $\prod_i K_{\nu_i}\backslash{L^+\BP_{\nu_i}}$-fiber bundle over $\nabla_n^{H, \wh Z_{\ul\nu}}\scrH^1$ (resp. $Hecke_{D, R_\ul\nu}^{\wh Z_{\ul\nu}}$ ) via the morphism  $\pi^{K_{\ul\nu}}$, and therefore locally of finite type. More precisely, since $\wh Z_{\nu_i,R_{\nu_i}}$ is projective (see Remark~\ref{RemFlagisquasiproj}), the scheme of morphisms $Mor(\wh Z_{\nu_i,R_{\nu_i}},\wh Z_{\nu_i,R_{\nu_i}})$ is of finite type by \cite[IV.4.c]{FGA} and consequently $K_{\nu_i} \backslash L^+\BP_{\nu_i} $ is of finite type.  Let $\wt W_i$ denote the affine Weyl group corresponding to $P_{\nu_i}$. Recall that by definition of the boundedness condition, the special fiber of $\wh Z_{\nu_i}$ is a finite union of closed affine Schubert varieties. Therefore the special fiber of the formal stack $[L^+\BP_{\nu_i}\backslash\wh Z_{\nu_i,R_{\nu_i}}]$ is a discrete set indexed by a finite subset of $\wt{W}_i$, corresponding to the $L^+\BP_{\nu_i}$-orbits, see \cite[Proposition 0.1]{Richarz}. Accordingly, we obtain a natural stratification $\{(\nabla_n^{H, Z_\ul\nu}\scrH_s^1)^\ul\lambda\}_\ul\lambda$ (resp. $\{\left(Hecke_{D,s}^{Z_\ul\nu}\right)^\ul\lambda\}_\ul\lambda$) of the special fiber $\nabla_n^{H, Z_{\ul\nu}}\scrH_s^1$ (resp. $Hecke_{D,s}^{Z_{\ul\nu}}$) of $\nabla_n^{H, \wh Z_{\ul\nu}}\scrH^1$ (resp. $Hecke_D^{Z_{\ul\nu}}$), indexed by $\ul\lambda \in \prod_i\wt W_i$, such that the incidence relation between strata is given by the obvious partial order on the product $\prod_i\wt{W}_i$, induced by the natural Bruhat order on each factor $\wt W_i$. The fiber over $\ul \lambda \in \prod_i\wt W_i$ appearing in this stratification is called \emph{Kottwitz-Rapoport stratum} corresponding to $\ul\lambda$. \\

\noindent

Assume that the special fiber $Z_\ul\nu$ of $\wh Z_\ul\nu$ equals the fiber product $\CS(\ul\omega)=\CS(\omega_1)\times_{\BF_q}\dots\times_{\BF_q} \CS(\omega_n)$ for $\ul\omega:=(\omega_i)\in \prod_i \wt W_i$. In particular, we have $\nabla_n^{H, Z_{\ul\nu}}\scrH_s^1=\left(\nabla_n^{H, \ul\omega}\scrH_s^1\right)^{\preceq \ul\omega}=\bigcup_{\ul\lambda\preceq\ul\omega}\left(\nabla_n^{H, \ul\omega}\scrH_s^1\right)^{\ul\lambda}$ (resp. $Hecke_{D,s}^{Z_{\ul\nu}}=\left(Hecke_{D,s}^{\ul\omega}\right)^{\preceq \ul\omega}:=\bigcup_{\ul\lambda\preceq \ul\omega} \left(Hecke_{D,s}^{\ul\omega}\right)^{\ul\lambda}$). 

\noindent
For each $\ul\omega$, let $\nabla_n^{H, \ul\omega}\scrH_s^{1^\circ}$ (resp. $Hecke_{D,s}^{\ul\omega^\circ}$) be the complement in $\nabla_n^{H, \ul\omega}\scrH_s^1$ (resp. $Hecke_{D,s}^{\ul\omega}$) of the union of all $\left(\nabla_n^{H, \ul\omega}\scrH_s^1\right)^{\ul\lambda}$ (resp. $\left(Hecke_{D,s}^{\ul\omega}\right)^{\ul\lambda}$) with  $\ul\lambda\prec\ul\omega$. We have the following

\begin{proposition}\label{PropDimOfNablaH}
Keep the above assumption about the bound $\wh Z_\ul\nu$. The stack $\nabla_n^{H, \ul\omega}\scrH_s^{1^\circ}$ (resp. $Hecke_{D,s}^{\ul\omega^\circ}$) is an smooth open and dense substack of $\nabla_n^{H, \ul\omega}\scrH_s^{1}$ (resp. $Hecke_{D,s}^{\ul\omega \circ}$) of dimension $\sum_{i=1}^n\ell_i(\omega_i)$ (resp. $d+\sum_{i=1}^n\ell_i(\omega_i)$). Here $\ell_i$ denotes the Bruhat-Chevally length function on $\wt W_i$ and $d=\dim \scrH_D^1(C,\FG)$.

\end{proposition}

\begin{proof}
This follows from \cite[Propposition 0.1]{Richarz}, Theorem \ref{ThmRapoportZinkLocalModel} (resp. proposition \ref{PropLocalModelHecke}) and the above discussion.
\end{proof}

\begin{theorem}\label{ThmFlatIntModel}

Keep the above notation. Assume that $\wh Z_\ul\nu$ admits a reduced representative. Furthermore assume that $G$ splits over a tamely ramified extension of $Q$ and that $p$ does not divide the order of the fundamental group of the derived group $G_{der}$. Then the moduli stack $\nabla_n^{H, \wh Z_\ul\nu}\scrH^1(C,\FG)^\ul\nu$ is flat over $\Spf R_{\wh Z_{\ul\nu}}$. In addition $\nabla_n^{H, \wh Z_\ul\nu}\scrH^1(C,\FG)^\ul\nu$ is normal.

\end{theorem}

\begin{proof}
We may assume that the boundness condition $\wh Z_{\nu_i}$ is defined over reflex ring $R_{\nu_i}=R_{\wh Z_{\nu_i}}$. This is because the extension of discrete valuation rings is faithfully flat. As $\nabla_n^{H, \wh Z_\ul\nu}\scrH^1(C,\FG)^\ul\nu$ is locally of finite type, the ring homomorphism from the stalk at a point $y\in \nabla_n^{H, \wh Z_\ul\nu}\scrH^1(C,\FG)^\ul\nu$  to the completion is faithfully flat. Therefore by Theorem \ref{ThmRapoportZinkLocalModel} it suffices to check the flatness of $\wh Z_\ul\nu \to \Spf A_\ul\nu$. Note that according to \cite[Theorem~8.4]{PR2} singularities of the special fiber of $\wh Z_{\nu_i}$ are rational. In particular the special fiber of $\wh Z_{\nu_i}$ is Cohen-Macaulay. Thus $\wh Z_{\nu_i}$ is Cohen-Macaulay by \cite[Theorem~17.3]{Matsumura}. Therefore $\wh Z_\ul\nu$ is Cohen-Macaulay by \cite{B-K}. On the other hand, as $R_{\wh Z_{\nu_i}}$ is a dvr and $S(\omega_i)$ is irreducible \cite[Remark~2.6]{Richarz13}, for every closed point $z \in \wh Z_{\nu_i}$, the dimension of\forget{the generic fiber of} $\CO_{\wh Z_{\nu_i},z} \otimes_{\wh A_{\nu_i}} \wh Q_{\nu_i}$ equals $\ell_i(\omega_i)=\dim \CO_{\wh Z_{\nu_i},z} \otimes_{\wh A_{\nu_i}} \kappa(\nu_i)$\forget{ dimension of the special fiber $\ell_i(\omega_i)$} \forget{comment{this is stack project lemma 10.124.9, should look for a better reference...}}. Hence we may conclude \forget{that the canonical morphism $A_\ul\nu \to \CO_{\wh Z_{\ul\nu},z}$ is flat }by \cite[IV, Proposition 6.1.5]{EGA} \forget{\comment{\cite[Theorem~23.1]{Matsumura}}}.
The assertion about normality now follows from \cite[IV, Th\'eor\`eme 12.2.4 and Proposition 6.14.1]{EGA}, \cite[Theorem~A]{Richarz13} and \cite[Theorem~8.4]{PR2}.

\end{proof}

Following Varshavsky's argument \cite[Corollary 2.21 c)]{Var}\forget{, together with the analogue of Rapoport-Zink local model Theorem \ref{ThmRapoportZinkLocalModel} and the above discussion,} we can now prove the following.

\begin{proposition}\label{PropICSheafNablaH}
Keep the assumption in Proposition~\ref{PropDimOfNablaH} The IC-sheaves $IC(\nabla_n^{H, \ul\omega}\scrH_s^1)$ and the restriction of $IC(Hecke_s^{\ul\omega})$ coincide up to some shift and Tate twists. In particular the restriction of  $IC(\nabla_n^{H, \ul\omega}\scrH_s^1)$ to each stratum is a direct sum of complexes of the form $\ol\BQ_\ell(k)[2k]$.

\end{proposition}

\begin{proof}

\forget
{
The proof proceeds in a similar way as \cite[Corollary 2.21 c)]{Var}, although one has to modify the following points. Namely, along the proof, instead of \cite[Theorem~2.20]{Var} one must implement the analogue of Rapoport-Zink local model which we obtained in Theorem \ref{ThmRapoportZinkLocalModel} (or alternatively Theorem \ref{ThmLocalModelI}). Furthermore, Bott-Samelson-Demazure resolution of singularities in the split reductive case, should be replaced with it's generalized version, constructed by Richarz \cite[Theorem~3.4]{Richarz}. Notice that also in his context, the Bott-Samelson-Demazure variety $\Sigma(\omega)$, which appears as a resolution of singularities of a Schubert variety $\CS(\omega)$ in a twisted affine flag variety, may also be viewed as an iterated tower of projective homogeneous fiber bundles; see \cite[Remark~2.9]{Richarz}. Recall that projective homogeneous varieties admit cellular decomposition (i.e. they  can be paved by affine spaces).  
}
 
According to Proposition~\ref{PropH_DL-Str} we may forget the $H$-level structure. The stratification on $\nabla_n^{\ul\omega}\scrH_s^1$ is induced by that of $Hecke_s^{\ul\omega}$. Moreover by Proposition \ref{PropLocalModelHecke} and Theorem \ref{ThmRapoportZinkLocalModel} and \cite[Remark~2.6]{Richarz} the smooth open stratum $\nabla_n^{\ul\omega}\scrH_s^{1^\circ}$  lies inside the pull back of the open smooth stratum  $Hecke_s^{\ul\omega^\circ}$ of $Hecke_s^{\ul\omega}$, thus the first argument is obvious over the open stratum.\\ 
\noindent
Regarding Proposition \ref{PropLocalModelHecke} and Theorem \ref{ThmRapoportZinkLocalModel}, we have the following diagram

\[
\xymatrix {
U_y\ar[r] \ar@/^1pc/[rr]^s \ar@/^2pc/[rrr]^{Smooth}& Hecke_{n,s}^{\ul\omega} & \wt{Hecke}_{n,s}^{\ul\omega} \ar[l]_\pi\ar[r]^f & \prod_i \CS(\omega_i) \ar@{=}[d] \\
U_{y}'\ar[r]\ar[u]_{i_{U_y}}\ar@/_1pc/[rr]_{s'}\ar@/_2pc/[rrr]_{Etale}&\nabla_n^{\ul\omega}\scrH_s^1\ar[u]_i & \wt{\nabla_n^{\ul\omega}\scrH_s}^1 \ar[l]_{\pi'}\ar[u] \ar[r]^{f'} & \prod_i \CS(\omega_i) \;.
}
\]

Now it suffices to show that the statement holds for the restriction of the IC-sheaves to the \'etale neighborhoods $U_y$ and $U_y'$

\begin{eqnarray*}
IC(U_y')=&(f'\circ s')^\ast IC(\prod_i \CS(\omega_i))\\
=&(f\circ s\circ i_{U_y})^\ast IC(\prod_i \CS(\omega_i))\\
=&i_{U_y}^\ast IC(U_y)(-dim \scrH^1/2)[-dim \scrH^1]\\
=&i_{U_y}^\ast IC(Hecke_{n,s}^{\ul\omega})|_{U_y}(-dim \scrH^1/2)[-dim \scrH^1]\\
=&i^\ast IC(Hecke_{n,s}^{\ul\omega})|_{U_y'}(-dim \scrH^1/2)[-dim \scrH^1],
\end{eqnarray*}

where the first equality follows from Theorem \ref{ThmRapoportZinkLocalModel}, the third equality  follows from Proposition \ref{PropLocalModelHecke} and all other equalities follow from commutativity of the above diagram.\\
Consider the Bott-Samelson-Demazure resolution of singularities $\Sigma(\omega_i)\to \CS(\omega_i)$ of the affine Schubert variety $\CS(\omega_i)$, constructed by Richarz \cite[Theorem~3.4]{Richarz}. The Bott-Samelson-Demazure variety $\Sigma(\omega_i)$ may be viewed as a tower of iterated projective homogeneous fiber bundles \cite[Remark~2.9]{Richarz}. Recall that projective homogeneous varieties admit cellular decomposition. Now the last assertion follows from the first statement of the corollary, Proposition~\ref{GR-Hecke} and the decomposition theorem \cite{BBD}.
\end{proof}

\begin{remark} \label{Lang Correspondence}
(The generalized Lang Correspondence)
We have the following roof from the local model diagram
\[ 
\xygraph{
!{<0cm,0cm>;<1cm,0cm>:<0cm,1cm>::}
!{(0,0) }*+{\nabla_n^{\wh Z_{\ul\nu}}\wt{\scrH_{R_\ul\nu}^1}}="a"
!{(-1.5,-1.5) }*+{\nabla_n^{\wh Z_{\ul\nu}}\scrH_{R_\ul\nu}^1}="b"
!{(1.5,-1.5) }*+{\prod_i \wh Z_{\nu_i, R_{\nu_i}},}="c"
"a":^{\pi}"b" "a":_{f}"c"
}  
\]
see Theorem~\ref{ThmRapoportZinkLocalModel}. 
Let $y:=\ul \CG$ be a global $\FG$-shtuka over $S$. Then, the above diagram together with the uniformization morphism $\Theta:=\Theta(\ul\CG)$ from \cite[Theorem 7.4]{AH_Global} induces the following roof
\[ 
\xygraph{
!{<0cm,0cm>;<1cm,0cm>:<0cm,1cm>::}
!{(0,0) }*+{\prod_{i} \ul{\CM}_{\ul{\BL}_0}^{\wh Z_{\nu_i}}\times_{\nabla_n^{\hat{Z}_{\ul\nu}}\scrH_{R_\ul\nu}^1}\nabla_n^{\wh Z_{\ul\nu}}\wt{\scrH_{R_\ul\nu}^1}}="a"
!{(-1.5,-1.5) }*+{\prod_{i} \ul{\CM}_{\ul{\BL}_0}^{\wh Z_{\nu_i}}}="b"
!{(1.5,-1.5) }*+{\prod_i \wh{Z}_i, R_{\nu_i}}="c"
"a":"b" "a":"c"
}  
\]
\forget{\comment{What does this fiber product mean? Is $\BF_q$ the right field here?}} Thus up to a choice of a section for $\pi$ we obtain a local morphism from the product of Rapoport-Zink spaces to $\prod_i \hat{Z}_i$. Note that $\prod_i \hat{Z}_i$ can be viewed as a parameter space for Hodge-Pink structures (see \cite{HartlPSp}). 


\end{remark}

%
%

\Verkuerzung
{
\vfill

\begin{minipage}[t]{0.5\linewidth}
\noindent
Esmail Arasteh Rad\\
Universit\"at M\"unster\\
Mathematisches Institut \\
erad@uni-muenster.de
\\[1mm]
\end{minipage}
\begin{minipage}[t]{0.45\linewidth}
\noindent
Somayeh Habibi\\
School of Mathematics,\\ 
Institute for Research in Fundamental Sciences
(IPM)
shabibi@ipm.ir
\\[1mm]
\end{minipage}
}
{}

\end{document}